\documentclass[a4paper,11pt]{amsart} 
\usepackage{amssymb,times, amscd,amsmath,amsthm, xypic}
\usepackage[all]{xy}
\usepackage{geometry}
\usepackage{mathrsfs}
\usepackage[dvips]{graphicx, color}
\author{Shoji Yokura}
\address
{Department of Mathematics and Computer Science, 
Faculty of Science, 
Kagoshima University, 21-35 Korimoto 1-chome, Kagoshima 890-0065, Japan}
\email {yokura@sci.kagoshima-u.ac.jp}
\title [Correspondences and characteristic classes of singular varieties]
{Enriched categories of correspondences and \\ characteristic classes of singular varieties}

\thanks {
\quad \emph{keywords} : algebraic cobordism, cobordism bicycle, correspondence, enriched category, characteristic classes }

\numberwithin{equation}{section}
\newtheorem{thm}[equation]{Theorem}
\newtheorem{pro}[equation]{Proposition}

\newtheorem{cor}[equation]{Corollary}
\newtheorem{lem}[equation]{Lemma}
\theoremstyle{definition}

\newtheorem{defn}[equation]{Definition}
\newtheorem{rem}[equation]{Remark}

\def\alp{\alpha}
\def\be{\beta}
\def\jeden{1\hskip-3.5pt1}

\def\bigstar{\mathbf{\star}}

\def\Cal{\mathscr}
\def\ga{\gamma}

\def \bQ{\mathbb Q}
\def\op{\operatorname}

\begin{document}




\begin{abstract} For the category $\mathscr V$ of complex algebraic varieties, the Grothen-dieck group 
of the commutative monoid of the isomorphism classes of correspondences $X \xleftarrow f M \xrightarrow g Y$ with proper morphism $f$ and smooth morphism $g$ (such a correspondence is called \emph{a proper-smooth correspondence}) gives rise to an enriched category $\mathscr Corr(\mathscr V)^+_{pro-sm}$ of proper-smooth correspondences.
In this paper we extend the well-known theories of characteristic classes of singular varieties such as Baum--Fulton--MacPherson's Riemann--Roch transformation (abbr. BFM--RR) and MacPherson's Chern class transformation and so on to this enriched category $\mathscr Corr(\mathscr V)^+_{pro-sm}$. In order to deal with local complete intersection (abbr. $\ell.c.i.$) morphism instead of smooth morphism, in a similar manner we consider an enriched category $\mathscr Zigzag(\mathscr V)^+_{pro-\ell.c.i.}$ of \emph{proper-$\ell.c.i.$} zigzags and extend BFM--RR to this 
category 
$\mathscr Zigzag(\mathscr V)^+_{pro-\ell.c.i.}$. We also consider an enriched category $\mathscr M_{*,*}(\mathscr V)^+_{\otimes}$ of proper-smooth correspondences $(X \xleftarrow f M \xrightarrow g Y; E)$ equipped with complex vector bundle $E$ on $M$ (such a correspondence is called \emph{a cobordism bicycle of vector bundle}) and we extend BFM--RR to this enriched category $\mathscr M_{*,*}(\mathscr V)^+_{\otimes}$ as well.
\end{abstract}

\maketitle


\section{Introduction}\label{intro} 

The algebraic cobordism $\Omega_*(X)$ of Levine and Morel \cite{LM} is generated by \emph{cobordism cycles}. A cobordism cycle is the isomorphism class of $$(M \xrightarrow f X; L_1, L_2, \cdots, L_r)$$ where $M$ is a quasi-projective smooth variety, $f:M \to X$ is a projective morphism and $L_i$'s are line bundles over $M$. In \cite{GK2} J.-L. Gonzal\'ez and K. Karu show that the above morphism $f:M \to X$ can be replaced by a \emph{proper morphism from a smooth variety $M$}. In order to obtain a bivariant-theoretic analogue $\mathbb B\Omega(X \xrightarrow f Y)$  of the algebraic cobordism $\Omega_*(X)$ in such a way that $\mathbb B\Omega(X \xrightarrow {} pt)$ is isomorphic to the algebraic cobordism $\Omega_*(X)$, in \cite{Yokura-obt} we introduce an \emph{oriented bivariant theory} $\mathbb {OB}(X \xrightarrow f Y)$, which is a bivariant theory in the sense of Fulton--MacPherson's bivariant theory \cite{FM} (see Remark \ref{remark} below). Note that Fulton--MacPherson's bivariant theory $\mathbb B(X \xrightarrow f Y)$ satisfies that $\mathbb B(X \xrightarrow f pt)$ is a covariant functor like a homology theory and $\mathbb B(X \xrightarrow {\op{id}_X} X)$ is a contravariant functor like a cohomology theory. 

$\mathbb {OB}(X \xrightarrow f Y)$ is generated by the cobordism cycles $(M \xrightarrow h X;L_1,L_2,\cdots,L_r)$ such that
\begin{enumerate}
\item $h:M \to X$ is a proper morphism
\item the composite $f \circ h:M \to Y$ is a smooth morphism.
\end{enumerate}
When $Y=pt$ is a point, then $M$ is smooth and $h:M \to X$ is proper, thus $\mathbb {OB}(X \xrightarrow {} pt)$ is generated by the isomorphism classes of  proper morphisms from smooth $M$ to $X$, thus the same as one considered in Gonzal\'ez--Karu's construction. The two morphisms $h:M \to X$ and $f \circ h:M \to Y$ are written as:
$$X \xleftarrow h M \xrightarrow {f \circ h} Y$$
In general, for a category $\mathscr C$ and for three objects $X, Y, M \in Obj(\mathscr C)$, $X \xleftarrow f M \xrightarrow g Y$ is called a \emph{correspondence} (span or roof) from $X$ to $Y$. The above correspondence $X \xleftarrow f M \xrightarrow g Y$ with proper morphism $f$ and smooth morphism $g$ shall be called \emph{a proper-smooth correspondence} from $X$ to $Y$, abusing words. In \cite{Yokura-bicycle} we introduce a \emph{cobordism bicycle of vector bundle} $(X \xleftarrow f M \xrightarrow g Y; E)$, a proper-smooth correspondence carrying a complex vector bundle $E$ on $M$, as a generalization or a \emph{bi-variant analogue} of \emph{cobordism cycle of vector bundle} $(M \xrightarrow f X;E)$ with $f:M \to X$ a proper morphism from a smooth variety $M$ and a complex vector bundle $E$ on $M$, introduced in a recent paper by Lee--Pandharipande \cite{LeeP}. In \cite{Yokura-bicycle} we discuss bivariant-theoretic properties and aspects of cobordism bicycles of complex vector bundles, but in this paper we will not treat such bivariant-theoretic aspects, instead we consider characteristic classes of singular varieties on proper-smooth correspondences and also on cobordism bicycles of complex vector bundles. 

A proper-smooth correspondence $X \xleftarrow f M \xrightarrow g Y$ can be considered as a morphism from $X$ to $Y$ as follows. Let $\op{Corr}(X,Y)_{pro-sm}$ be the set of all proper-smooth correspondences from $X$ to $Y$. Then the composition
$$\circ :\op{Corr}(X,Y)_{pro-sm}\times \op{Corr}(Y,Z)_{pro-sm} \to \op{Corr}(X,Z)_{pro-sm}$$ defined by
$$(X \xleftarrow f M \xrightarrow g Y) \circ (Y \xleftarrow h N \xrightarrow k Z): = X \xleftarrow {f \circ \widetilde{h}} M \times_Y N \xrightarrow {k \circ \widetilde {g}} Z$$
is well-defined because the pullback $\widetilde h$ of a proper morphism $h$ is proper and the pullback $\widetilde g$ of  a smooth morphism $g$ is smooth, and the composite of two proper morphisms is proper and the composite of smooth morphisms is smooth. Here we consider the following commutative diagram and the middle square is the fiber product:
\begin{equation}\label{product}
\xymatrix{
&& M\times_Y N \ar [dl]_{\widetilde{h}} \ar[dr]^{\widetilde{g}} &&\\
& M \ar [dl]_{f} \ar [dr]^{g} && N \ar [dl]_{h} \ar[dr]^{k}\\
X & &  Y && Z }
\end{equation}

Then from the category $\mathscr V$ of complex algebraic varieties we get the following category $\op{Corr}(\mathscr V)_{pro-sm}$ of proper-smooth correspondences:
\begin{itemize}
\item $Obj(\op{Corr}(\mathscr V)_{pro-sm})=Obj(\mathscr V)$,
\item For two objects $X$ and $Y$, $hom_{\op{Corr}(\mathscr V)_{pro-sm}}(X,Y) = \op{Corr}(X,Y)_{pro-sm}$
\end{itemize}
\begin{rem} For a recent higher-categorical study of correspondences (in derived algebraic geometry), see Gaitsgory--Rozenblyum's book \cite{GR1} (cf. \cite{GR2}), in particular Chapter 5 - Chapter 9.
\end{rem}

On the category $\mathscr V$ of complex algebraic varieties, let us consider
Baum--Fulton--MacPherson's Todd class transformation (or Riemann--Roch transformation) \cite{BFM} $td_*^{\op{BFM}}:G_0(-) \to H_*(-)\otimes \mathbb Q$, which is a \emph{unique} natural transformation from the covariant functor $G_0(-)$ of Grothendieck groups of coherent sheaves to the covariant Borel--Moore homology theory $H_*(-)\otimes \mathbb Q$ with rational coefficients, satisfying the ``smooth condition'' that if $X$ is smooth, then the value $td_*^{\op{BFM}}(\mathcal O_X)$ of the structure sheaf $\mathcal O_X$ is equal to $td(TX) \cap [X]$, the Poincar\'e dual of the total Todd cohomology class $td(TX)$ of the tangent bundle $TX$. Here we remark that the uniqueness of the transformation $td_*^{\op{BFM}}:G_0(-) \to H_*(-)\otimes \mathbb Q$ is due to the above smooth condition. Then these two functors $G_0(-)$ and $H_*(-)\otimes \mathbb Q$ and the the natural transformation $td_*^{\op{BFM}}:G_0(-) \to H_*(-)\otimes \mathbb Q$ can be naturally extended to the category $\op{Corr}(\mathscr V)_{pro-sm}$ of proper-smooth correspondences as follows:
\begin{pro} Define the functors $\mathcal G_0:\op{Corr}(\mathscr V)_{pro-sm} \to \mathscr Ab$ and $\mathcal H^{Todd}_*:\op{Corr}(\mathscr V)_{pro-sm} \to \mathscr Ab$ as follows:
\begin{enumerate}
\item For an object $X \in Obj(\op{Corr}(\mathscr V)_{pro-sm}) = Obj(\mathscr V)$, 
$$\mathcal G_0(X) =G_0(X), \quad \mathcal H^{Todd}_*(X) :=H_*(X)\otimes \mathbb Q.$$
\item For a morphism  
$$X \xleftarrow f M \xrightarrow g Y \in hom_{\op{Corr}(\mathscr V)_{pro-sm}}(X,Y) = \op{Corr}(X,Y)_{pro-sm},$$
\begin{align*}
\mathcal G_0(X \xleftarrow f M \xrightarrow g Y) &:=f_*g^*:\mathcal G_0(Y) \to \mathcal G_0(X), \\
\mathcal H^{Todd}_*(X \xleftarrow f M \xrightarrow g Y) &:= f_* \bigl(td(T_g)\cap g^* \bigr):\mathcal H^{Todd}_*(Y)  \to \mathcal H^{Todd}_*(X) .
\end{align*}
\end{enumerate}
Then $\mathcal G_0:\op{Corr}(\mathscr V)_{pro-sm} \to \mathscr Ab$ and $\mathcal H^{Todd}_*:\op{Corr}(\mathscr V)_{pro-sm} \to \mathscr Ab$  are 
functors\footnote{Since the correspondence $X \xleftarrow f  M \xrightarrow g Y$ is considered as a morphism from $X$ to $Y$ and the homomorphism $\mathcal  F(X \xleftarrow f  M  \xrightarrow g Y):=f_*g^*:  \mathcal F(Y) \to \mathcal F(X)$ is a homomorphism from $\mathcal F(Y)$ to $\mathcal F(X)$, $\mathcal  F:\op{Corr}(\mathscr V) \to \mathscr  Ab$ should be called a contravariant functor. But, since it satisfies $\mathcal F(\alp \circ \be) = \mathcal  F(\alp)\circ \mathscr F(\be)$, we call it a functor, i.e., a covariant functor.} in the sense of $\mathcal G_0(\alp \circ \be) = \mathcal G_0(\alp) \circ \mathcal G_0(\be)$ and $\mathcal H^{Todd}_*(\alp \circ \be) = \mathcal H^{Todd}_*(\alp) \circ \mathcal H^{Todd}_*(\be)$, and Baum--Fulton--MacPherson's Todd class transformation $td_*^{\op{BFM}}: G_0(-) \to H_*(-)\otimes \mathbb Q$ is extended to the natural transformation
$$td_*^{\op{BFM}}: \mathcal  G_0(-) \to \mathcal H^{Todd}_*(-).$$
\end{pro}
Furthermore, the set of isomorphism classes of proper-smooth correspondences becomes an Abelian monoid by taking the disjoint union or direct sum $\sqcup$, i.e.,
$$[X \xleftarrow {f_1} M_1 \xrightarrow {g_1} Y] + [X \xleftarrow {f_2} M_2 \xrightarrow {g_2} Y] := [X \xleftarrow {f_1 \sqcup f_2} M_1 \sqcup M_2 \xrightarrow {g_1 \sqcup g_2} Y].$$
Its Grothendieck group shall be denoted by $\op{Corr}(X,Y)_{pro-sm}^+$ . Then the above ``product'' $ \circ :\op{Corr}(X,Y)_{pro-sm}\times \op{Corr}(Y,Z)_{pro-sm}\to \op{Corr}(X,Z)_{pro-sm}$ is extended to 
$$\circ :\op{Corr}(X,Y)_{pro-sm}^+ \times \op{Corr}(Y,Z)_{pro-sm}^+ \to \op{Corr}(X,Z)_{pro-sm}^+.$$
Using this we can define the following  $\mathscr Ab$-enriched (or preadditive\footnote{This cannot be replaced by ``additive" because there does not exist a zero object in the category  $\mathscr Corr(\mathscr V) _{pro-sm}^+$.}) category $\mathscr  Corr(\mathscr V) _{pro-sm}^+$  associated to such correspondences:
\begin{itemize}
\item $Obj(\mathscr  Corr(\mathscr V) _{pro-sm}^+)=Obj(\mathscr V)$.
\item For two objects $X,Y$, $hom_{\mathscr  Corr(\mathscr V) _{pro-sm}^+}(X,Y) := \op{Corr}(X,Y)_{pro-sm}^+$.
\end{itemize}
Then we have the following theorem:
\begin{thm}
We define $\mathscr G_0: \mathscr  Corr^+_{pro-sm} \to \mathscr  Ab$ and $\mathscr  H^{Todd}_*: \mathscr  Corr^+_{pro-sm} \to \mathscr  Ab$ as follows: 
\begin{enumerate}
\item For an object $X \in Obj(\mathscr  Corr^+_{pro-sm})$, 
$$\text{$\mathscr  G_0(X):= G_0(X)$ and $\mathscr  H^{Todd}_*(X):=H_*(X)\otimes \mathbb Q$,}$$
\item For a morphism $\sum_in_i[X \xleftarrow {f_i} M_i \xrightarrow {g_i} Y]  \in hom_{\mathscr  Corr_{pro-sm}^+}(X,Y) := \op{Corr}(X,Y)_{pro-sm}^+$ 
\begin{align*}
& \mathscr  G_0 \Bigl (\sum_in_i[X \xleftarrow {f_i} M_i \xrightarrow {g_i} Y] \Bigr) := \sum_in_i(f_i)_*(g_i)^*:\mathscr G_0(Y) \to \mathscr G_0(X),  \\
& \mathscr  H^{Todd}_*\Bigl (\sum_in_i[X \xleftarrow {f_i} M_i \xrightarrow {g_i} Y] \Bigr) \\
& \qquad \qquad \qquad := \sum_in_i(f_i)_* \Bigl(td(T_{g_i})\cap (g_i)^* \Bigr):\mathscr H^{Todd}_*(Y)  \to \mathscr H^{Todd}_*(X) .
\end{align*}
\end{enumerate}
Then $\mathscr G_0: \mathscr  Corr^+_{pro-sm} \to \mathscr  Ab$ and $\mathscr  H^{Todd}_*: \mathscr  Corr^+_{pro-sm} \to \mathscr  Ab$ are both 
functors in the sense of $\mathscr G_0(\alp \circ \be) = \mathscr G_0(\alp) \circ \mathscr G_0(\be)$ and $\mathscr H^{Todd}_*(\alp \circ \be) = \mathscr H^{Todd}_*(\alp) \circ \mathscr H^{Todd}_*(\be)$ and Baum--Fulton--MacPherson's Todd class transformation $td_*^{\op{BFM}}: G_0(-) \to H_*(-)\otimes \mathbb Q$ is extended to the natural transformation
$$td_*^{\op{BFM}}: \mathscr  G_0(-) \to \mathscr  H^{Todd}_*(-).$$
\end{thm}
\begin{rem} To avoid some possible confusion, we make a remark about notations or symbols concerning the Grothendieck group or the group completion of an Abelian monoid $A$. Let $(A, \dot +)$ be an Abelian monoid. Then its group completion or the Grothendieck group $K(A)$ has two constructions and one construction is defined by
$K(A):=F(A)/E(A),$
where $F(A)$ is the free abelian group generated by the set $A$, i.e.,
$$F(A):=\Bigl \{ \sum n_a a \Bigl | \, a \in A, n_a \in \mathbb Z, n_a = 0 \, \, \text{for almost all $a$'s} \Bigr \}$$
and the group operation $+$ on $F(A)$ is defined by
$$\sum n_a a  + \sum m_a a  := \sum (n_a + m_a) a$$
and $E(A)$ is the following subgroup of $F(A)$:
\begin{align*}
E(A) := \Bigl \{ \sum n_{a,b} & \bigl (a + b + (-1)(a \dot +  b) \bigr ) \\
& \Bigl | \, a, b \in A, n_{a,b} \in \mathbb Z, n_{a,b} = 0 \, \, \text{for almost all $a$'s and $b$'s} \Bigr \}.
\end{align*}
The equivalence class $a + E(A) \in K(A)$ of $a \in F(A)$ should be denoted by $[a]$ and the group operation on $F(A)$ should be denoted by, say, $\widehat +$,  and
$[a] \, \widehat + \, [b] := [a + b].$
Then, since $K(A) =F(A)/E(A)$, $[a+b] =[a \dot +  b]$ in $K(A)$, thus we have 
$[a] \, \widehat + \, [b]=[a \dot + b].$
In this way, the monoid operation $\dot +$ on $A$ is extended to a group operation $\widehat +$ on $K(A)$.
To avoid many symbols or notation, usually $[a] \, \widehat + \, [b]$ is simply denoted by $a + b$ unless some confusion is possible. Thus, in the above theorem, an element $\sum_in_i[X \xleftarrow {f_i} M_i \xrightarrow {g_i} Y]  \in \op{Corr}(X,Y)_{pro-sm}^+$ should be denoted by
$\sum_in_i \bigl [[X \xleftarrow {f_i} M_i \xrightarrow {g_i} Y] \bigr ],$
but we follow the above convention.
\end{rem}
In a similar manner we can extend the other well-know characteristic classes of singular varieties (\cite{Mac}, \cite{BSY}) to the enriched category $\mathscr  Corr(\mathscr V) _{pro-sm}^+$.  
Furthermore, similarly, we can define \emph{a proper-local complete intersection (abbr. $\ell.c.i.$) correspondence}. But, in this case, since the pullback of a $\ell.c.i$-morphism is not necessarily a $\ell.c.i.$-morphism, thus an argument similar to the above does not work. To remedy this drawback, we consider \emph{a proper-$\ell.c.i.$ zigzag}, which is a finite sequence of proper-$\ell.c.i.$ correspondence $X \xleftarrow f M \xrightarrow g Y$ with proper $f$ and $\ell.c.i.$ morphism $g$. 
This will be discussed in \S 4. In \S 5 we will define the enriched category $\mathscr M_{*,*}(\mathscr V)^+_{\otimes}$ of cobordism bicycles of vector bundles and extend Baum--Fulton--MacPherson's Riemann--Roch transformation or Todd class transformation to this enriched category:
\begin{thm} 
\begin{enumerate}
\item Let us define $\mathscr G_0^{\otimes}: \mathscr M_{*,*}(\mathscr V)^+_{\otimes} \to \mathscr Ab$ by
\begin{enumerate}
\item For an object $X$, $\mathscr G_0^{\otimes}(X) =G_0(X)$,
\item For a morphism $\sum_i n_i[X \xleftarrow {p_i} V_i \xrightarrow {s_i} Y; E_i] \in hom_{\mathscr M_{*,*}(\mathscr V)^+_{\otimes}}(X,Y) =\mathcal  M_{*,*}(X,Y)^+_{\otimes}$,
\begin{align*}
\mathscr G_0^{\otimes}(& \sum_i n_i[X \xleftarrow {p_i} V_i \xrightarrow {s_i} Y; E_i]):= \\
& \hspace{1cm} \sum_i n_i (p_i)_*  ([E_i] \otimes (s_i)^* ): \mathscr G_0^{\otimes}(Y) \to \mathscr G_0^{\otimes}(X).
\end{align*}
\end{enumerate}
Then $\mathscr G_0^{\otimes}: \mathscr M_{*,*}(\mathscr V) ^+_{\otimes} \to \mathscr Ab$ is a 
functor in the sense of 
$\mathscr G_0^{\otimes}(\alp \circ_{\otimes} \be) = \mathscr G_0^{\otimes}(\alp)  \circ \mathscr G_0^{\otimes}(\be).$
\item For the Todd class $td$ and the Chern character $ch$, we define 

$\mathscr H^{td, ch}_*: \mathscr M_{*,*}(\mathscr V)^+_{\otimes} \to \mathscr Ab$ by
\begin{enumerate}
\item $\mathscr H^{td, ch}_*(X) := H_*(X)\otimes \mathbb Q$,
\item For {\small $\sum_i n_i[X \xleftarrow {p_i} V_i \xrightarrow {s_i} Y; E_i] \in hom_{\mathscr M_{*,*}(\mathscr V)^+_{\otimes}}(X,Y)=\mathcal  M_{*,*}(X,Y)^+_{\otimes}$},
\begin{align*}
& \mathscr H^{td, ch}_* (\sum_i n_i[X \xleftarrow {p_i} V_i \xrightarrow {s_i} Y; E_i]):= \\
& \sum_i n_i {p_i}_*\bigl(td(T_{s_i}) \cap ch(E_i) \cap (s_i)^* \bigr):\mathscr H^{td, ch}_*(Y) \to \mathscr H^{td, ch}_*(X).
\end{align*}
\end{enumerate}
Then $\mathscr H^{td, ch}_*: \mathscr M_{*,*}(\mathscr V)^+_{\otimes} \to \mathscr Ab$ is a  
functor in the sense of 
$\mathscr H^{td, ch}_*(\alp \circ_{\otimes} \be) = \mathscr H^{td, ch}_*(\alp) \circ \mathscr H^{td, ch}_*(\be).$
\item  Baum--Fulton--MacPherson's Todd class transformation $td_*^{\op{BFM}}$ gives rise to the natural transformation of these two 
functors $\mathscr G_0^{\otimes}: \mathscr M_{*,*}(\mathscr V)^+_{\otimes} \to \mathscr Ab$ and $\mathscr H^{td, ch}_*: \mathscr M_{*,*}(\mathscr V)^+_{\otimes} \to \mathscr Ab: td_*^{\op{BFM}}: \mathscr G_0^{\otimes}(-) \to \mathscr H^{td, ch}_*(-).$
\end{enumerate}
\end{thm}

\begin{rem}\label{remark} In \cite{LS} P. Lowrey and T. Sch\"urg have constructed a derived algebraic cobordism $d\Omega_*(X)$ for derived schemes. In \cite{An} 
  T. Annala has obtained a bivariant-theoretic version $\Omega^*(X \to Y)$ of Levine--Morel's algebraic cobordism, using the construction of Lowrey and Sch\"urg and the construction of a universal bivariant theory of the author \cite{Yokura-obt}. Furthermore, in \cite{AY} (also see \cite{An2}) T. Annala and the author have constructed a bivariant-theoretic version of Lee--Pandharipande's algebraic cobordism of vector bundles \cite{LeeP}.
\end{rem}

\section{Enriched Categories of correspondences}

\begin{defn} Let $\mathscr V$ be the category of complex algebraic varieties. For any pair $(X,Y)$ of complex algebraic varieties $X$ and $Y$, the set of all correspondences $X \xleftarrow f V \xrightarrow g Y$ is denoted by $\op{Corr}(X,Y)$:
$$\op{Corr}(X,Y) := \bigl \{X \xleftarrow f V \xrightarrow g Y \, \, | \, \,  f \in hom_{\mathscr V}(V, X), g \in hom_{\mathscr V}(V, Y) \bigr \}.$$
\end{defn}
\begin{defn} The \emph{category $\op{Corr}(\mathscr V)$ of correspondences of $\mathscr V$} is defined by:
\begin{enumerate}
\item $Obj(\op{Corr}(\mathscr V))=Obj(\mathscr V).$
\item For two objects $X$ and $Y$, the set of homomorphisms from $X$ to $Y$ is defined to be $\op{Corr}(X,Y)$, i.e.,
$hom_{\op{Corr}(\mathscr V)}(X,Y)=\op{Corr}(X,Y),$
where the composition 
\begin{equation}\label{product*}
\circ :\op{Corr}(X,Y) \times \op{Corr}(Y,Z) \to \op{Corr}(X,Z)
\end{equation}
is defined by, for $(X \xleftarrow f M \xrightarrow g Y) \in \op{Corr}(X,Y)$ and $(Y \xleftarrow h N \xrightarrow k Z) \in \op{Corr}(Y,Z)$,
$$(X \xleftarrow f M \xrightarrow g Y) \circ (Y \xleftarrow h N \xrightarrow k Z): = X \xleftarrow {f \circ \widetilde{h}} M \times_Y N \xrightarrow {k \circ \widetilde {g}} Z,$$
where we use the diagram (\ref{product}).
\end{enumerate}
\end{defn}
\begin{defn} Let $\mathscr Ab$ be the category of abelian groups. A functor $E: \mathscr V \to \mathscr Ab$ is called a  \emph{bifunctor}\footnote{Usually a bifunctor or bi-functor is used for a functor $E:\mathscr C^{op} \times \mathscr C' \to \mathscr B$ defined on the Cartesian product of two categories, which is contravariant with respect to the first factor and covariant with respect to the second factor.} if it is both a covariant and a contravariant functor.
\end{defn}
\begin{lem}\label{lem1} Let $E:\mathscr V \to \mathscr Ab$ be a bifunctor.
We define $\mathcal  E:\op{Corr}(\mathscr V) \to \mathscr  Ab$ by 
\begin{enumerate}
\item For an object $X \in Obj(\op{Corr}(\mathscr V)) = Obj(\mathscr V)$, $\mathcal  E(X) = E(X)$.
\item For a morphism $(X \xleftarrow f  V  \xrightarrow g Y) \in hom_{\op{Corr}(\mathscr V)}(X,Y) =Corr(X,Y)$, 
$$\mathcal  E(X \xleftarrow f  V  \xrightarrow g Y):=f_*g^*:  \mathcal E(Y) \to \mathcal E(X).$$
\end{enumerate}
If $E$ satisfies the base change formula, i.e., for any fiber square (left square) the following diagram (the right square) commutes: 
$$\xymatrix{
A' \ar[d]_{\widetilde h} \ar[r]^{\widetilde g} & A \ar[d]^{h \quad \qquad \Longrightarrow \quad \qquad} \\
B' \ar[r]_g & B,
}
\xymatrix{
E(A') \ar[d]_{(\widetilde h)_*} && E(A) \ar[ll]_{(\widetilde g)^*} \ar[d]^{h_*}\\
E(B') && E(B) \ar[ll]^{g^*}, 
}
$$
 then $\mathcal  E:\op{Corr}(\mathscr V) \to \mathscr  Ab$ is a  
 functor in the sense of
$\mathcal E(\alp \circ \be) = \mathcal  E(\alp)\circ \mathcal E(\be)$
 for 
 $\alp \in hom_{\op{Corr}(\mathscr V)}(X,Y) =\op{Corr}(X,Y)$ and $\be \in hom_{\op{Corr}(\mathscr V)}(Y,Z) =\op{Corr}(Y,Z)$.
\end{lem}
\begin{proof} Let $\alp = (X \xleftarrow f M \xrightarrow g Y) \in \op{Corr}(X,Y)$ and $\be = (Y \xleftarrow h N \xrightarrow k Z) \in \op{Corr}(Y,Z)$.
Then 
$\alp \circ \be = X \xleftarrow {f \circ \widetilde{h}} M \times_Y N \xrightarrow {k \circ \widetilde {g}} Z$ (see the diagram (\ref{product})).
Hence 
\begin{align*}
\mathcal E(\alp \circ \be) & = \mathcal E(X \xleftarrow {f \circ \widetilde{h}} M \times_Y N \xrightarrow {k \circ \widetilde {g}} Z) \\
& = (f\circ\widetilde h)_*(k \circ \widetilde g)^*\\
& = f_* ((\widetilde h)_*(\widetilde g)^*) k^* \\
& = f_* (g^*h_*) k^* \quad \text{(by the above base change formula)}\\
& = (f_* g^*)\circ (h_*k^*) \\
& = \mathcal  E(\alp)\circ \mathcal E(\be).
\end{align*}
\end{proof}
\begin{rem} We remark that if a bifunctor $E:\mathscr V \to \mathscr Ab$ satisfies the above base change formula for a fiber square, then it satisfies the ``isomorphism condition'' that for an isomorphism $h:M \xrightarrow {\cong} M'$
$$h^*h_* = \op{id}_{E(M)}, \quad h_*h^* = \op{id}_{E(M')}.$$
These follow from the fact that the following two squares are fiber squares:
$$
\xymatrix{
M \ar[d]_{\op{id}_M}^{=} \ar[r]^{\op{id}_M}_{=} & M \ar[d]^{h}_{\cong} \\
M  \ar[r]_{h}^{\cong} & M'
}
\qquad \qquad 
\xymatrix{
M \ar[d]_{h}^{\cong} \ar[r]^{h}_{\cong} & M' \ar[d]^{\op{id}_{M'}}_{=} \\
M'  \ar[r]_{\op{id}_{M'}}^{=} & M'
}
$$
\end{rem} 

Suppose that we have two bifunctors $E_1, E_2:\mathscr V \to \mathscr  Ab$ satisfying the base change formula described above and a natural transformation $\tau:E_1 \to E_2$, which means that it is a natural transformation form the covariant functor $E_1$ to the covariant functor $E_2$ and at the same time it is a natural transformation form the contravariant functor $E_1$ to the contravariant functor $E_2$. Then we get two functors associated to the bifunctors $E_1, E_2:\mathscr V \to \mathscr  Ab$
$$\mathcal E_1, \mathcal E_2: \op{Corr}(\mathscr V) \to \mathscr  Ab.$$
Since we have the following commutative diagram for a correspondence $X \xleftarrow f M \xrightarrow g Y$, 
\begin{equation}\label{diagram-1}
\xymatrix{
E_1(X) \ar[d]_{\tau}  & E_1(M) \ar[l]_{f_*} \ar[d]_{\tau}   & E_1(Y) \ar[l]_{g^*} \ar[d]^{\tau} \\
E_2(X)  & E_2(M) \ar[l]^{f_*} & E_2(Y) \ar[l]^{g^*}, 
}
\end{equation}
the natural transformation $\tau:E_1 \to E_2$ give rise to the associated natural transformation between two functors $\mathcal E_1, \mathcal E_2: \op{Corr}(\mathscr V) \to \mathscr  Ab$:
$$\mathcal T: \mathcal E_1 \to \mathcal E_2.$$
Namely, for an object $X$, $\mathcal T: \mathcal E_1(X) \to \mathcal E_2(X)$ is nothing but the homomorphism $\tau:\mathcal E_1(X) \to \mathcal E_2(X)$. For a morphism
$(X \xleftarrow f M \xrightarrow g Y) \in hom_{\op{Corr}(\mathscr V)}(X,Y) =\op{Corr}(X,Y)$, $\mathcal T: \mathcal E_1 \to \mathcal E_2$ being a natural transformation means that the following diagram commutes:
\begin{equation*}
\xymatrix{
\mathcal E_1(X) \ar[d]_{\tau}  && \mathcal E_1(Y) \ar[ll]_{f_*g^*} \ar[d]^{\tau} \\
\mathcal E_2(X)  & &  \mathcal E_2(Y) \ar[ll]^{f_*g^*}, 
}
\end{equation*}
which is nothing but the above commutative diagram (\ref{diagram-1}).
\begin{rem}  We define the following subcategories $\op{Corr}(\mathscr V)_{-,\op{id}} , \op{Corr}(\mathscr V)_{\op{id}, -}$ of $\op{Corr}(\mathscr V)$:
\begin{align*}
\op{Corr}(\mathscr V)_{-,\op{id}}(X,Y) & := \op{Corr}(X,Y)_{-,\op{id}} \\
& = \{X \xleftarrow f Y \xrightarrow {\op{id}_Y} Y \, | \, f \in hom_{\mathscr V}(Y,X) \} \cong hom_{\mathscr V}(Y,X).
\end{align*}
\begin{align*}
\op{Corr}(\mathscr V)_{\op{id}, -}(X,Y) & := \op{Corr}(X,Y)_{\op{id}, -} \\
& = \{X \xleftarrow {\op{id}_X} X \xrightarrow g  Y \, | \, g \in hom_{\mathscr V}(X,Y) \} \cong  hom_{\mathscr V}(X,Y). 
\end{align*}
Then, if we restrict the functor $\mathcal E: \op{Corr}(\mathscr V) \to \mathscr Ab$ to these two subcategories, then
$\mathcal E: \op{Corr}(\mathscr V)_{-,\op{id}} \to \mathscr Ab$ is the same as the covariant functor $E: \mathscr V \to \mathscr Ab$ and $\mathcal E: \op{Corr}(\mathscr V)_{\op{id}, -} \to \mathscr Ab$ is the same as the contravariant functor $E: \mathscr V \to \mathscr Ab$.
\end{rem}

Two correspondences $X \xleftarrow {f_1} M_1 \xrightarrow {g_1} Y$ and $X \xleftarrow {f_2} M_2 \xrightarrow {g_2} Y$ in the category $\mathscr V$ are called isomorphic if there exists an isomorphism $h:M_1 \cong M_2$ such that the following diagram commutes:
$$\xymatrix{
&M_1 \ar[dl]_{f_1} \ar[dd]_{h}^{\cong} \ar [dr]^{g_1} &\\
X & & Y\\
&M_2 \ar[ul]^{f_2} \ar [ur]_{g_2} &}$$
The isomorphism class is denoted by $[X \xleftarrow f M \xrightarrow g Y]$ and the set of isomorphism classes of such correspondences becomes an Abelian monoid by taking the disjoint union $\sqcup$ , i.e.,
$$[X \xleftarrow {f_1} M_1 \xrightarrow {g_1} Y] + [X \xleftarrow {f_2} M_2 \xrightarrow {g_2} Y] := [X \xleftarrow {f_1 \sqcup f_2} M_1 \sqcup M_2 \xrightarrow {g_1 \sqcup g_2} Y].$$
The definition is well-defined, i.e., it does not depend on the choice of representatives. The group completion of this Abelian monoid, i.e., the Grothendieck group of a commutative monoid, is denoted by
$\op{Corr}(X,Y)^+$.
We observe that the product of correspondences (Definition \ref{product*}):
$$\circ: \op{Corr}(X,Y) \times \op{Corr}(Y,Z) \to \op{Corr}(X,Z)$$
can be extended to the Grothendieck group $\op{Corr}(X,Y)^+$, i.e., we have the following bilinear product:
\begin{lem} For three varieties $X,Y, Z$ we have the bilinear map
$$\circ: \op{Corr}(X,Y)^+\times \op{Corr}(Y,Z)^+ \to \op{Corr}(X,Z)^+$$
defined by
\begin{align*}
\Bigl (\sum_i n_i[X \xleftarrow {f_i} M_i \xrightarrow {g_i} Y] \Bigr)  & \circ \Bigl (\sum_j m_j[Y \xleftarrow {h_j} N_j \xrightarrow {k_j} Z]\Bigr) \\
& :=\sum_{i,j}n_im_j[(X \xleftarrow {f_i} M_i \xrightarrow {g_i} Y) \circ (Y \xleftarrow {h_j} N_j \xrightarrow {k_j} Z)].
\end{align*}
\end{lem}
\begin{proof} 
(i) First we show that the following
$$[X \xleftarrow f M \xrightarrow g Y] \circ [Y \xleftarrow h N \xrightarrow k Z]:= [(X \xleftarrow f M \xrightarrow g Y) \circ (Y \xleftarrow h N\xrightarrow k Z)]$$
is well-defined, i.e., it is independent of the choice of a representative. Namely, $(X \xleftarrow f M \xrightarrow g Y) \cong (X \xleftarrow {f'} M' \xrightarrow {g'} Y)$ and $(Y \xleftarrow h N \xrightarrow k Z) \cong (Y \xleftarrow {h'} N' \xrightarrow {k'} Z)$ imply that
$$(X \xleftarrow f M \xrightarrow g Y) \circ (Y \xleftarrow h N\xrightarrow k Z) \cong (X \xleftarrow {f'} M' \xrightarrow {g'} Y) \circ (Y \xleftarrow {h'} N' \xrightarrow {k'} Z).$$
This isomorphism follows from the universality of the fiber product and from the following commutative diagrams:
$$\xymatrix{
&& M\times_Y N \ar [dl]_{\widetilde{h}} \ar[dr]^{\widetilde{g}} &&\\
& M \ar [dl]_{f} \ar [dr]^{g} && N \ar [dl]_{h} \ar[dr]^{k}\\
X & &  Y && Z\\
& M' \ar [ul]^{f'} \ar [ur]_{g'} && N' \ar [ul]^{h'} \ar[ur]_{k'}\\
&& M'\times_Y N' \ar [ul]^{\widetilde{h'}} \ar[ur]_{\widetilde{g'}} &&
 }
$$
(ii) Next, in order to be able to extend the product $\circ$ to the Grothendieck group, we need to show that with respect to each factor, the disjoint sum is distributive, i.e., we show the following distributivity (below $+$ indicates the addition defined by the disjoint sum as defined above):
\begin{align*}
& \Bigl ([X \xleftarrow {f_1} M_1 \xrightarrow {g_1} Y]  + [X \xleftarrow {f_2} M_2 \xrightarrow {g_2} Y] \Bigr) \circ [Y \xleftarrow h N \xrightarrow k Z] \\
& = [X \xleftarrow {f_1} M_1 \xrightarrow {g_1} Y] \circ [Y \xleftarrow h N \xrightarrow k Z] + [X \xleftarrow {f_2} M_2 \xrightarrow {g_2} Y] \circ [Y \xleftarrow h N \xrightarrow k Z],
\end{align*}
\begin{align*}
& [X  \xleftarrow {f}  M \xrightarrow {g} Y] \circ \Bigl ([Y \xleftarrow {h_1} N_1 \xrightarrow {k_1} Z] + [Y \xleftarrow {h_2} N_2 \xrightarrow {k_2} Z] \Bigr ) \\
& = [X \xleftarrow {f} M \xrightarrow {g} Y] \circ [Y \xleftarrow {h_1} N_1 \xrightarrow {k_1} Z]  + [X \xleftarrow {f} M \xrightarrow {g} Y] \circ [Y \xleftarrow {h_2} N_2 \xrightarrow {k_2} Z].
\end{align*}
We show only the first one, since it is the same for the second one.
Namely (see the diagrams below), it suffices to show the following:
\begin{align*}
[X \xleftarrow {(f_1 \sqcup f_2) \circ \overline h } & (M_1 \sqcup M_2) \times_Y N  \xrightarrow {k\circ \overline {g_1 \sqcup g_2}} Z] \\
& = [X \xleftarrow {f_1 \circ \widetilde h }  M_1 \times_Y N \xrightarrow {k\circ \widetilde {g_1}} Z] + [X \xleftarrow {f_2 \circ \widetilde {\widetilde h} } M_2 \times_Y N \xrightarrow {k\circ \widetilde {g_2}} Z]\\
&= [X \xleftarrow {f_1 \circ \widetilde h \sqcup  f_2 \circ \widetilde {\widetilde h}}  \bigl (M_1 \times_Y N \bigr) \sqcup \bigl (M_2 \times_Y N \bigr) \xrightarrow {k\circ \widetilde {g_1} \sqcup  k \circ \widetilde {g_2}} Z]\\
&= [X \xleftarrow {(f_1 \sqcup f_2) \circ (\widetilde h \sqcup \widetilde {\widetilde h})}  \bigl (M_1 \times_Y N \bigr) \sqcup \bigl (M_2 \times_Y N \bigr) \xrightarrow {k\circ (\widetilde {g_1} \sqcup \widetilde {g_2})} Z]
\end{align*}
$$
\xymatrix{
&& M_1\times_Y N \ar [dl]_{\widetilde{h}} \ar[dr]^{\widetilde{g_1}} &&\\
& M_1 \ar [dl]_{f_1} \ar [dr]^{g_1} && N \ar [dl]_{h} \ar[dr]^{k}\\
X & &  Y && Z,  }
$$
$$
\xymatrix{
&& M_2\times_Y N \ar [dl]_{\widetilde {\widetilde{h}}} \ar[dr]^{\widetilde{g_2}} &&\\
& M_2 \ar [dl]_{f_2} \ar [dr]^{g_2} && N \ar [dl]_{h} \ar[dr]^{k}\\
X & &  Y && Z, }
$$
$$
\xymatrix{
&& (M_1 \sqcup M_2) \times_Y N \ar [dl]_{\overline {h}} \ar[dr]^{\overline {g_1 \sqcup g_2}} &&\\
& M_1 \sqcup M_2 \ar [dl]_{f_1 \sqcup f_2} \ar [dr]^{g_1 \sqcup g_2} && N \ar [dl]_{h} \ar[dr]^{k}\\
X & &  Y && Z. }
$$

By the universality of the fiber product it follows that $(M_1 \sqcup M_2) \times_Y N \cong \bigl (M_1 \times_Y N \bigr) \sqcup \bigl (M_2 \times_Y N \bigr)$, i.e., the fiber product commutes with the disjoint union $\sqcup$, and that $\overline h= \widetilde h \sqcup \widetilde {\widetilde h}$ and $\overline {g_1 \sqcup g_2}= \widetilde {g_1} \sqcup \widetilde {g_2}$, i.e., the pullback commutes with the disjoint union $\sqcup$. Hence we have that 
\begin{align*}
[X \xleftarrow {(f_1 \sqcup f_2)  \circ \overline h } & (M_1 \sqcup M_2) \times_Y N   \xrightarrow {k\circ \overline {g_1 \sqcup g_2}} Z] \\
&= [X \xleftarrow {(f_1 \sqcup f_2) \circ (\widetilde h \sqcup \widetilde {\widetilde h})}  \bigl (M_1 \times_Y N \bigr) \sqcup \bigl (M_2 \times_Y N \bigr) \xrightarrow {k\circ (\widetilde {g_1}  \sqcup \widetilde {g_2})} Z].
\end{align*}
\end{proof}
\begin{cor} The Abelian group $\op{Corr}(X,X)^+$ becomes a ring with zero $[X \xleftarrow {} \emptyset \xrightarrow {} X]$ and  unit $[X \xleftarrow {\op{id}_X} X \xrightarrow {\op{id}_X} X]$ by the above product.
\end{cor}
\begin{rem}The ring $\op{Corr}(X,X)^+$ is called \emph{the Hecke ring of $X$} or \emph{the ring of correspondences of $X$} in Michio Kuga's article \cite[\S 2 Toy Hecke Rings and Toy Zeta Functions]{Kuga} (cf.\cite[Appendix, \S 4 Hecke Rings]{Kuga2}).
\end{rem}
\begin{defn} The following category $\mathscr Corr(\mathscr V)^+$ is called \emph{the $\mathscr Ab$-enriched (or preadditive) category of correspondences}:
\begin{enumerate}
\item $Obj(\mathscr Corr(\mathscr V)^+) = Obj(\mathscr V).$
\item For two objects $X$ and $Y$, $hom_{\mathscr Corr(\mathscr V)^+}(X,Y) = \op{Corr}(X,Y)^+.$
\end{enumerate}
\end{defn}
Here is an enriched category version of Lemma \ref{lem1}.
\begin{lem}\label{lem2} Let $E:\mathscr V \to \mathscr Ab$ be a  
bifunctor such that 
\begin{enumerate}
\item[(i)] it satisfies the base change formula for fiber products, i.e., for a fiber square (the left square) the following diagram (the right square) commutes: 
$$\xymatrix{
A' \ar[d]_{\widetilde h} \ar[r]^{\widetilde g} & A \ar[d]^{h \quad \qquad \Longrightarrow \quad \qquad} \\
B' \ar[r]_g & B,
}
\xymatrix{
E(A') \ar[d]_{(\widetilde h)_*} && E(A) \ar[ll]_{(\widetilde g)^*} \ar[d]^{h_*}\\
E(B') && E(B) \ar[ll]^{g^*}, 
}
$$ 
\item[(ii)] it is additive with respect to the disjoint union $\sqcup$ in the sense that 
$$i^* \oplus (i')^*: E(X \sqcup X') \cong E(X) \oplus E(X')$$ 
where $i:X \to X \sqcup X'$ and $i':X \to X \sqcup X'$ are inclusions, and
\item[(iii)] it is functorial with pushforward $f_*$ and pullback $g^*$.
\end{enumerate}
We define $\mathscr  E:\mathscr Corr(\mathscr V)^+ \to \mathscr  Ab$ by 
\begin{enumerate}
\item For an object $X \in Obj(\mathscr Corr(\mathscr V)^+) = Obj(\mathscr V)$, $\mathscr  E(X) = E(X)$.
\item For $\sum_i n_i[X \xleftarrow {f_i}  M_i \xrightarrow {g_i} Y] \in hom_{\mathscr Corr(\mathscr V)^+ }(X,Y) =\op{Corr}(X,Y)^+$, 
\begin{equation}\label{grot}
\mathscr  E \Bigl (\sum_i n_i[X \xleftarrow {f_i}  M_i \xrightarrow {g_i} Y] \Bigr):=\sum_i n_i(f_i)_*(g_i)^*:  \mathscr E(Y) \to \mathscr E(X).
\end{equation}
\end{enumerate}
Then $\mathscr  E:\mathscr Corr(\mathscr V)^+ \to \mathscr  Ab$ is a functor in the sense of
$\mathscr  E(\alp \circ \be) = \mathscr  E(\alp)\circ \mathscr  E(\be)$.
\end{lem}
\begin{proof} 
Before starting the proof, we observe that the above definition (\ref{grot})
can be interpreted as a group homomorphism
$$\mathscr E: \op{Corr}(X,Y)^+ \to Hom (E(Y), E(X)),$$
where $Hom (E(Y), E(X))$ is the abelian group of all the homomorphisms from $E(Y)$ to $E(X)$.
We also observe that for the inclusions $i_k:M_k \to M_1 \sqcup M_2$ ($k=1,2$), we have 
$$(i_k)^*(i_k)_* = \op{id}_{E(M_k)} (k=1,2), \quad (i_2)^*(i_1)_*= (i_1)^*(i_2)_* = 0$$
 by the following fiber squares:
$$\xymatrix{
M_k \ar[d]_{\op{id}_{M_k}} \ar[r]^{\op{id}_{M_k}} & M_k \ar[d]^{i_k} \\
M_k \ar[r]_{i_k} & M_1 \sqcup M_2,\quad 
}
\xymatrix{
\emptyset \ar[d] \ar[r] & M_1 \ar[d]^{i_1} \\
M_2 \ar[r]_{i_2}  & M_1 \sqcup M_2, \quad 
}
\xymatrix{
\emptyset \ar[d] \ar[r] & M_2 \ar[d]^{i_2} \\
M_1 \ar[r]_{i_1}  & M_1 \sqcup M_2,
}
$$ 
(i) First we show that $\mathscr  E([X \xleftarrow f M \xrightarrow g Y]):=f_*g^*$ is well-defined, namely, it is independent of the choice of a representative, i.e., $(X \xleftarrow f M \xrightarrow {g} Y) \cong (X \xleftarrow {f'} M' \xrightarrow {g'} Y)$ implies $f_*g^* = (f')_*(g')^*.$
Since there exists an isomorphism $h:M \cong M'$ such that the following diagram commutes:
$$\xymatrix{
&M \ar[dl]_{f} \ar[dd]_{h}^{\cong} \ar [dr]^{g} &\\
X & & Y\\
&M' \ar[ul]^{f'} \ar [ur]_{g'} &}$$
\begin{align*}
f_*g^* &=  (f'\circ h)_*(g'\circ h)^*\\
& = (f')_*h_*\circ h^*(g')^*\\
& =(f')_*(h_*\circ h^*)(g')^*\\
& = (f')_*(g')^* \quad \text{(by the isomorphism condition $h_* h^*=\op{id}_{E(M')}$).}
\end{align*}
(ii) Next we show that the definition (\ref{grot}) is well-defined on the Grothendieck group $\op{Corr}(X,Y)^+$. For that it suffices to show that the map
$$ \mathscr E_0:\op{Corr}(X,Y) \to Hom (E(Y), E(X))$$
defined by $\mathscr E_0 ([X \xleftarrow f M \xrightarrow g Y]) := f_* g^*$ is a homomorphism of \emph{monoids}, i.e., it satisfies that
\begin{align*}
\mathscr E_0 ([X \xleftarrow {f_1} M_1 \xrightarrow {g_1} Y] & + [X \xleftarrow {f_2} M_2 \xrightarrow {g_2} Y]) \\
& = 
\mathscr E_0 ([X \xleftarrow {f_1} M_1 \xrightarrow {g_1} Y])  + \mathscr E_0 ([X \xleftarrow {f_2} M_2 \xrightarrow {g_2} Y]),
\end{align*}
namely
\begin{align*}
\mathscr E_0 ([X \xleftarrow {f_1 \sqcup  f_2}  M_1 \sqcup M_2 & \xrightarrow {g_1 \sqcup g_2} Y] ) \\
& =  \mathscr E_0 ([X \xleftarrow {f_1} M_1 \xrightarrow {g_1} Y])  + \mathscr E_0 ([X \xleftarrow {f_2} M_2 \xrightarrow {g_2} Y]),
\end{align*}
i.e., 
\begin{equation}\label{equ2}
 (f_1 \sqcup  f_2)_*(g_1 \sqcup g_2 )^* = (f_1)_*(g_1)^* + (f_2)_*(g_2 )^*.
 \end{equation}
Because, as well-known, it follows from the universality of the Grothendieck group that 
$$\mathscr  E :  \op{Corr}(X,Y)^+ \to Hom(E(Y), E(X))$$
is a unique extension of the above monoid homomorphism $ \mathscr E_0:\op{Corr}(X,Y) \to Hom (E(Y), E(X))$:
$$\xymatrix{
\op{Corr}(X,Y) \ar[dr]_{\gamma} \ar[rr]^{\mathscr E_0} && Hom(E(Y), E(X)) \\
 & \op{Corr}(X,Y)^+ \ar[ur]_{\mathscr E}, }
$$
where $\gamma:\op{Corr}(X,Y)   \to \op{Corr}(X,Y)^+$ maps an element $x$ to the equivalence class $[x]$.
Here is a proof of the above equality (\ref{equ2}): Consider the following commutative diagram:
$$\xymatrix{
&& E(M_1 \sqcup M_2)\ar[dd]^{(i_1)^* \oplus (i_2)^*} \ar[dd]_{\cong}  \ar[dll]_{(f_1 \sqcup f_2)_*}   &&\\
E(X)  &&&& E(Y) \ar[ull]_{(g_1 \sqcup g_2)_*}  \ar[dll]^{(g_1)^* \oplus (g_2)^*} \\
& & E(M_1) \oplus E(M_2) \ar[ull]^{(f_1)_* pr_1 + (f_2)_*pr_2 \quad \quad }  }
$$
Here we note that 
\begin{equation}\label{inverse}
\text{the inverse of $(i_1)^* \oplus (i_2)^*$ \, \, is \, \,  $(i_1)_* pr_1 + (i_2)_*pr_2.$}
\end{equation}
Then we have
\begin{align*}
& (f_1)_*(g_1)^* + (f_2)_*(g_2)^* \\
& = \Bigl ((f_1)_* pr_1 + (f_2)_*pr_2 \Bigr) \circ \Bigl ((g_1)^* \oplus (g_2)^* \Bigr )\\
& = \Bigl ((f_1 \sqcup f_2)_* \circ \bigl ((i_1)_* pr_1 + (i_2)_*pr_2 \bigr ) \Bigr) \circ \Bigl (\bigl ((i_1)^* \oplus (i_2)^* \bigr ) \circ (g_1 \sqcup g_2)_* \Bigr )\\
& = (f_1 \sqcup f_2)_* \circ \Bigl (\bigl ((i_1)_* pr_1 + (i_2)_*pr_2 \bigr ) \circ ((i_1)^* \oplus (i_2)^* \bigr ) \Bigr ) \circ (g_1 \sqcup g_2)_* \\
& = (f_1 \sqcup f_2)_* \circ (g_1 \sqcup g_2)_* 
\end{align*}
The proof of (\ref{inverse}) is the following: 
\begin{itemize}
\item $((i_1)^* \oplus (i_2)^*)\circ \bigl ((i_1)_* pr_1 + (i_2)_*pr_2) \bigr) = \op{id}_{E(M_1)\oplus E(M_2) }$: Indeed, for $(x,y) \in E(M_1)\oplus E(M_1)$,
\begin{align*}
& \Bigl (((i_1)^* \oplus (i_2)^*)\circ ((i_1)_* pr_1 + (i_2)_*pr_2) \Bigr) (x,y) \\
& = ((i_1)^* \oplus (i_2)^*)\circ ((i_1)_*x + (i_2)_*y) )\\
& = ((i_1)^*(i_1)_*x, (i_2)^*(i_2)_*y) \\
& =(x,y) \quad \quad \text{(since $(i_1)^*(i_1)_*= \op{id}_{E(M_1)}$ and $(i_2)^*(i_2)_*= \op{id}_{E(M_2)}$)}
\end{align*}
Hence we have $((i_1)^* \oplus (i_2)^*)\circ ((i_1)_* pr_1 + (i_2)_*pr_2) = \op{id}_{E(M_1)\oplus E(M_2) }$.
\item $((i_1)_* pr_1 + (i_2)_*pr_2)) \circ ((i_1)^* \oplus (i_2)^*)= \op{id}_{E(M_1 \sqcup M_2) }$: It is clear that
$$((i_1)_* pr_1 + (i_2)_*pr_2)) \circ ((i_1)^* \oplus (i_2)^*) =(i_1)_* (i_1)^* + (i_2)_*(i_2)^*.$$
It is more or less clear that $(i_1)_* (i_1)^* + (i_2)_*(i_2)^*= \op{id}_{E(M_1 \sqcup M_2)}$, but this can be seen as follows:
We can write any element $a$ in $E(M_1 \sqcup M_2)$ as $a= (i_1)_*x + (i_2)_*y$ where $(x,y) \in E(M_1)\oplus E(M_2)$. Then
\begin{align*}
& \Bigl( (i_1)_* (i_1)^*  + (i_2)_*(i_2)^* \Bigr)(a) \\
& = \Bigl( (i_1)_* (i_1)^* + (i_2)_*(i_2)^* \Bigr)(a) \\
& = \Bigl( (i_1)_* (i_1)^* + (i_2)_*(i_2)^* \Bigr)((i_1)_*x + (i_2)_*y) \\
& = (i_1)_* (i_1)^*(i_1)_*x + (i_2)_*(i_2)^*(i_1)_*x + (i_1)_* (i_1)^*(i_2)_*y + (i_2)_*(i_2)^*(i_2)_*y\\
& = (i_1)_*x + (i_2)_*y  \\
& \, \, \, \, \text{(since $(i_1)^*(i_1)_*= \op{id}$,$(i_2)^*(i_1)_*= (i_1)^*(i_2)_*=0$, $(i_2)^*(i_2)_*= \op{id}$)}\\
& =a.
\end{align*}
Hence we have $((i_1)_* pr_1 + (i_2)_*pr_2)) \circ ((i_1)^* \oplus (i_2)^*)= \op{id}_{E(M_1 \sqcup M_2) }$.
 \end{itemize}

\end{proof}
Suppose that we have two bifunctors $E_1, E_2:\mathscr V \to \mathscr  Ab$ satisfying the base change formula for fiber products and a natural transformation $\tau:E_1 \to E_2$. Then, in the same way as in the previous section, we get two functors
$$\mathscr E_1, \mathscr E_2: \mathscr Corr(\mathscr V)^+  \to \mathscr  Ab$$
and a natural transformation
$\mathscr T: \mathscr E_1 \to \mathscr E_2.$
\section{Enriched categories of proper-smooth correspondences and \\
characteristic classes of singular varieties}

In this section we consider proper-smooth correspondences $X \xleftarrow f M \xrightarrow g Y$, i.e., $f:M \to X$ is a proper morphism and and $g:M \to Y$ is a smooth morphism. The set of all such proper-smooth correspondences from $X$ to $Y$ is denoted by $\op{Corr}(X,Y)_{pro-sm}$. As explained in the introduction, the following product is well-defined:
$$\circ: \op{Corr}(X,Y)_{pro-sm} \times \op{Corr}(Y,Z)_{pro-sm} \to \op{Corr}(X,Z)_{pro-sm}$$
$$(X \xleftarrow f M \xrightarrow g Y) \circ (Y \xleftarrow h N \xrightarrow k Z): = X \xleftarrow {f \circ \widetilde{h}} M \times_Y N \xrightarrow {k \circ \widetilde {g}} Z,$$
where we use the diagram (\ref{product}).

\begin{defn}
Let $\op{Corr}(\mathscr V)_{pro-sm}$ be the category of proper-smooth correspondences. Namely
\begin{enumerate}
\item $Obj(\op{Corr}(\mathscr V)_{pro-sm}) = Obj(\mathscr V)$,
\item For two objects $X, Y$, $hom_{\op{Corr}(\mathscr V)_{pro-sm}}(X,Y) = \op{Corr}(X,Y)_{pro-sm}$. The composition 
$$ \circ :hom_{\op{Corr}(\mathscr V)_{pro-sm}}(X,Y) \times hom_{\op{Corr}(\mathscr V)_{pro-sm}}(Y,Z) \to hom_{\op{Corr}(\mathscr V)_{pro-sm}}(X,Z)$$
is 
\, $ \circ :\op{Corr}(X,Y)_{pro-sm}\times \op{Corr}(Y,Z)_{pro-sm}\to \op{Corr}(X,Z)_{pro-sm}$.
\end{enumerate}
\end{defn}
The set of isomorphism classes of proper-smooth correspondences becomes an Abelian monoid by taking the disjoint sum  and its group completion shall be denote by $\op{Corr}(X,Y)_{pro-sm}^+$ . Then the above composition (or product) \,\,\, $ \circ :\op{Corr}(X,Y)_{pro-sm}\times \op{Corr}(Y,Z)_{pro-sm}\to \op{Corr}(X,Z)_{pro-sm}$ is extended to 
$$\circ :\op{Corr}(X,Y)_{pro-sm}^+ \times \op{Corr}(Y,Z)_{pro-sm}^+ \to \op{Corr}(X,Z)_{pro-sm}^+.$$
\begin{rem}\label{M+} When $Y$ is a point $pt$, $\op{Corr}(X,pt)_{pro-sm}^+$ is the same as $\mathcal M^+(X)$ defined in \cite{LM, GK2}.
\end{rem}
Using this we can define the following  $\mathscr Ab$-enriched (or preadditive) category associated to such correspondences:
\begin{defn}
\begin{enumerate}
\item $Obj(\mathscr  Corr(\mathscr V) _{pro-sm}^+)=Obj(\mathscr V)$.
\item For two objects $X,Y$, $hom_{\mathscr  Corr(\mathscr V) _{pro-sm}^+}(X,Y) := \op{Corr}(X,Y)_{pro-sm}^+$.
\end{enumerate}
\end{defn}

Baum--Fulton--MacPherson's Todd class (or the singular Riemann--Roch transformation) \cite{BFM} is a unique natural transformation $td_*^{\op{BFM}}:G_0(-) \to H_*(-)\otimes \mathbb Q$ from the covariant functor $G_0(-)$ of Grothendieck groups of coherent sheaves to the covariant Borel--Moore homology theory $H_*(-)\otimes \mathbb Q$ with rational coefficients such that for a smooth variety $X$ the value $td_*^{\op{BFM}}(\mathcal O_X)= td(TX) \cap [X]$ the Poincar\'e dual of the total Todd class $td(TX)$ of the tangent bundle $TX$, where $\mathcal O_X$ is the structure sheaf of $X$. 

For a proper maps $f:M \to X$ we have the commutative diagram:
\begin{equation}\label{diagram*2}
\xymatrix{
G_0(X)\ar[d]_{td_*^{\op{BFM}}}  && G_0(M) \ar[ll]_{f_*} \ar[d]^{td_*^{\op{BFM}}} \\
H_*(X)\otimes \mathbb Q && H_*(M)\otimes \mathbb Q \ar[ll]^{f_*}.
}
\end{equation}
For a smooth morphism $g:M \to Y$ (see Remark \ref{rem-*} below), 
we have the following Verdier--Riemann--Roch formula \cite{Ver} (see \cite[Conjecture, p.137]{BFM} and \cite[Theorem 18.2]{Fulton-book}), i.e., we have the commutative diagram:\begin{equation}\label{diagram*3}
\xymatrix{
G_0(M)\ar[d]_{td_*^{\op{BFM}}}  && G_0(Y) \ar[ll]_{g^*} \ar[d]^{td_*^{\op{BFM}}} \\
H_*(M)\otimes \mathbb Q && H_*(Y)\otimes \mathbb Q. \ar[ll]^{td(T_g) \cap g^*}
}
\end{equation}
Here $td(T_g)$ is the total Todd class of the relative tangent bundle $T_g$ of the smooth morphism $g$. Hence combining the above two commutative diagrams (\ref{diagram*2}) and (\ref{diagram*3}), for a correspondence $(X \xleftarrow f M \xrightarrow g Y)$ with proper morphism $f$ and smooth morphism $g$, we have the commutative diagrams:
\begin{equation}\label{diagram***4}
\xymatrix{
G_0(X)\ar[d]_{td_*^{\op{BFM}}}  && G_0(M) \ar[ll]_{f_*} \ar[d]_{td_*^{\op{BFM}}} && G_0(Y) \ar[ll]_{g^*} \ar[d]^{td_*^{\op{BFM}}} \\
H_*(X)\otimes \mathbb Q && H_*(M)\otimes \mathbb Q \ar[ll]^{f_*} && H_*(Y)\otimes \mathbb Q. \ar[ll]^{td(T_g) \cap g^*}
}
\end{equation}
\begin{pro}\label{pro-BFM} Define the functors $\mathcal G_0:\op{Corr}(\mathscr V)_{pro-sm} \to \mathscr Ab$ and $\mathcal H^{Todd}_*:\op{Corr}(\mathscr V)_{pro-sm} \to \mathscr Ab$ as follows:
\begin{enumerate}
\item For an object $X \in Obj(\op{Corr}(\mathscr V)_{pro-sm}) = Obj(\mathscr V)$, 
$$\mathcal G_0(X) =G_0(X), \quad \mathcal H^{Todd}_*(X) :=H_*(X)\otimes \mathbb Q.$$
\item For 
$X \xleftarrow f M \xrightarrow g Y \in hom_{\op{Corr}(\mathscr V)_{pro-sm}}(X,Y) = \op{Corr}(X,Y)_{pro-sm}$, 
\begin{align*}
\mathcal G_0(X \xleftarrow f M \xrightarrow g Y) &:=f_*g^*:\mathcal G_0(Y) \to \mathcal G_0(X),  \\
\mathcal H^{Todd}_*(X \xleftarrow f M \xrightarrow g Y) &:= f_* \bigl(td(T_g)\cap g^* \bigr):\mathcal H^{Todd}_*(Y)  \to \mathcal H^{Todd}_*(X).
\end{align*}
\end{enumerate}
Then $\mathcal G_0:\op{Corr}(\mathscr V)_{pro-sm} \to \mathscr Ab$ and $\mathcal H^{Todd}_*:\op{Corr}(\mathscr V)_{pro-sm} \to \mathscr Ab$  are 
functors in the sense of $\mathcal G_0(\alp \circ \be) = \mathcal G_0(\alp) \circ \mathcal G_0(be)$ and $\mathcal H^{Todd}_*(\alp \circ \be) = \mathcal H^{Todd}_*(\alp) \circ \mathcal H^{Todd}_*(be)$, and Baum--Fulton--MacPherson's Todd class transformation $td_*^{\op{BFM}}: G_0(-) \to H_*(-)\otimes \mathbb Q$ is extended to the natural transformation
$$td_*^{\op{BFM}}: \mathcal  G_0(-) \to \mathcal H^{Todd}_*(-).$$
\end{pro}
\begin{proof} Since the naturality of the transformation $td_*^{\op{BFM}}: \mathcal  G_0(-) \to \mathcal H^{Todd}_*(-)$ follows the above commutative diagrams (\ref{diagram***4}), we only have to show the covariance of the functors $\mathcal G_0$ and $\mathcal H^{Todd}_*$.

\begin{enumerate}
\item The functoriality of $\mathcal  G_0(-)$: It is well-known (e.g., see \cite[Lemma 29.5.2 (Flat base change)]{Stack} that the covariant functor $G_0(-)$ of Grothendieck groups of coherent sheaves satisfies the base change formula. Hence, as in the proof of Lemma \ref{lem1}, we see that $\mathcal G_0:\op{Corr}(\mathscr V)_{pro-sm} \to \mathscr Ab$ is a 
functor in the sense of $\mathcal G_0 (\alp \circ \be) = \mathcal G_0(\alp) \circ \mathcal G_0(\be)$.
\item The functoriality of $\mathcal H^{Todd}_*(-)$: First we remark that it is known that the Borel--Moore homology theory\footnote{In an earlier version of the present paper we use the Chow homology theory instead of the Borel--Moore homology theory, since it satisfies the base change formula \cite[Proposition 1.7 and Theorem 1.7]{Fulton-book}. However, J. Sch\"urmann pointed out that the Borel--Moore homology theory satisfies the base change formula.} also satisfies the base change formula (cf. \cite{CG}).

For $\alp = (X \xleftarrow f M \xrightarrow g Y)$ and $\be = (Y \xleftarrow h N  \xrightarrow k Z)$, we have $\alp \circ \be  =  X \xleftarrow {f\circ \widetilde h} M \times_Y N \xrightarrow {k \circ \widetilde g} Z$ (see Definition \ref{product}). Hence
\begin{align*}
& \mathcal H^{Todd}_* (\alp \circ \be)  = \mathcal H^{Todd}_*(X \xleftarrow {f \circ \widetilde{h}} M \times_Y N \xrightarrow {k \circ \widetilde {g}} Z)\\
& = (f \circ \widetilde{h})_* \Bigl (td(T_{k \circ \widetilde {g}}) \cap (k \circ \widetilde {g})^* \Bigr) \\
& = f_* (\widetilde{h})_* \Bigl (\bigl (td(T_{\widetilde g} \cup td((\widetilde g)^*T_k) \bigr) \cap (\widetilde {g})^*k^* \Bigr) \\
& \hspace{2cm} \quad \text{(since $td$ is multiplicative, thus $td(T_{k \circ \widetilde {g}}) = td(T_{\widetilde g}) \cup td((\widetilde g)^*T_k)$)}\\
& = f_* (\widetilde{h})_* \Bigl (td(T_{\widetilde {g}})\cap (\widetilde {g})^*td(T_k) \cap (\widetilde {g})^*k^* \Bigr) \\
& = f_* (\widetilde{h})_* \Bigl (td(T_{\widetilde {g}})\cap (\widetilde {g})^*\bigl (td(T_k) \cap k^* \bigr ) \Bigr) \\
& \hspace{5cm} \quad \text{(by \cite[Theorem 3.2. (d) (Pull-back)]{Fulton-book})}\\
& = f_* (\widetilde{h})_* \Bigl (td((\widetilde {h})^*T_g)\cap (\widetilde {g})^*\bigl (td(T_k) \cap k^* \bigr ) \Bigr) \quad \text{(since $T_{\widetilde {g}} = (\widetilde {h})^*T_g$) }\\
& = f_* (\widetilde{h})_* \Bigl ((\widetilde {h})^*td(T_g)\cap (\widetilde {g})^*\bigl (td(T_k) \cap k^* \bigr ) \Bigr) \\
& = f_* \Bigl (td(T_g)\cap (\widetilde {h})_*(\widetilde {g})^*\bigl (td(T_k) \cap k^* \bigr ) \Bigr) \quad \text{(by the projection formula)} \\
\end{align*}
\begin{align*} 
& = f_* \Bigl (td(T_g)\cap g^*h_*\bigl (td(T_k) \cap k^* \bigr ) \Bigr) \\
& \hspace{4cm} \quad \text{(by  the base change formula $(\widetilde{h})_* (\widetilde {g})^*= g^*h_*$)} \\ 
&  = f_* \Bigl (td(T_g)\cap g^* \Bigl (h_*\bigl (td(T_k) \cap k^* \bigr ) \Bigr) \Bigr) \\
&  = \Bigl (f_* (td(T_g)\cap g^*) \Bigr)  \Bigl (h_*\bigl (td(T_k) \cap k^* \bigr ) \Bigr) \\
& = \mathcal H^{Todd}_*(\alp) \circ \mathcal H^{Todd}_*(\be).
\end{align*}
\end{enumerate}
\end{proof}
\begin{rem} The proof above only uses that the Todd class $c\ell= td$ is a cohomological characteristic class of isomorphism classes of complex algebraic vector bundles, which is functorial for pullbacks $g^*$ and multiplicative, i.e., 
$c\ell(F) = c\ell(F') \cup cl(F'')$
for a short exact sequence $0 \to F' \to F \to F'' \to 0$ of complex algebraic vector bundles.\footnote{Note that if we consider this short exact sequence $0 \to F' \to F \to F'' \to 0$ as that of topological complex vector bundles, forgetting algebraic structure, then it splits, i.e., $F \cong F' \oplus F''$}
 In the proof above we have a short exact sequence $0 \to T_{\widetilde g} \to T_{k \circ \widetilde {g}} \to (\widetilde g)^*T_k \to 0$.
\end{rem}


\begin{thm}\label{thm-BFM}
We define $\mathscr G_0: \mathscr  Corr^+_{pro-sm} \to \mathscr  Ab$ and $\mathscr  H^{Todd}_*: \mathscr  Corr^+_{pro-sm} \to \mathscr  Ab$ as follows: 
\begin{enumerate}
\item For an object $X \in Obj(\mathscr  Corr^+_{pro-sm})$, 

$\mathscr  G_0(X):= G_0(X)$ and $\mathscr  H^{Todd}_*(X):=H_*(X)\otimes \mathbb Q$,
\item For a morphism $\sum_in_i[X \xleftarrow {f_i} M_i \xrightarrow {g_i} Y]  \in hom_{\mathscr  Corr_{pro-sm}^+}(X,Y) := \op{Corr}(X,Y)_{pro-sm}^+, $ 
$$\mathscr  G_0 \Bigl ( \sum_in_i[X \xleftarrow {f_i} M_i \xrightarrow {g_i} Y] \Bigr) := \sum_in_i(f_i)_*(g_i)^*:\mathscr  G_0(Y) \to \mathscr  G_0(X),$$
\begin{align*}
\mathscr  H^{Todd}_*\Bigl (  \sum_in_i[ & X \xleftarrow {f_i}  M_i \xrightarrow {g_i} Y] \Bigr) \\
&:= \sum_in_i(f_i)_* \Bigl(td(T_{g_i})\cap (g_i)^* \Bigr):\mathscr  H^{Todd}_*(Y) \to \mathscr  H^{Todd}_*(X).
\end{align*}
\end{enumerate}
Then $\mathscr G_0: \mathscr  Corr^+_{pro-sm} \to \mathscr  Ab$ and $\mathscr  H^{Todd}_*: \mathscr  Corr^+_{pro-sm} \to \mathscr  Ab$ are 
functors in the sense of $\mathscr G_0(\alp \circ \be) = \mathscr G_0(\alp) \circ \mathscr G_0(be)$ and $\mathscr H^{Todd}_*(\alp \circ \be) = \mathscr H^{Todd}_*(\alp) \circ \mathscr H^{Todd}_*(be)$ and Baum--Fulton--MacPherson's Todd class transformation $td_*^{\op{BFM}}: G_0(-) \to H_*(-)\otimes \mathbb Q$ is extended to the natural transformation
$$td_*^{\op{BFM}}: \mathscr  G_0(-) \to \mathscr  H^{Todd}_*(-).$$
\end{thm}
\begin{proof}
We note that $G_0(-)$and $H_*(-)\otimes \mathbb Q$ 
are additive bifunctors with respect to pushforward $f_*$ for proper morphisms $f$ and pullbacks $g^*$ for smooth morphisms $g$. Then it suffices to show the following equalities:
\begin{align*}
\mathscr  G_0 \Bigl ([X \xleftarrow {f_1 \sqcup f_2}&   M_1 \sqcup M_2  \xrightarrow  {g_1 \sqcup g_2} Y] \Bigr) \\
& = \mathscr  G_0 \Bigl ([X \xleftarrow {f_1} M_1 \xrightarrow {g_1} Y] \Bigr) + \mathscr  G_0 \Bigl ( [X \xleftarrow {f_2} M_2 \xrightarrow {g_2} Y] \Bigr), 
\end{align*}
\begin{align*}
\mathscr  H^{Todd}_*\Bigl ([ &X \xleftarrow {f_1 \sqcup f_2} M_1 \sqcup M_2  \xrightarrow {g_1 \sqcup g_2} Y] \Bigr)\\
& = \mathscr  H^{Todd}_* \Bigl ([X \xleftarrow {f_1} M_1 \xrightarrow {g_1} Y] \Bigr) +  \mathscr  H^{Todd}_* \Bigl ( [X \xleftarrow {f_2} M_2 \xrightarrow {g_2} Y] \Bigr).
\end{align*}
Namely,
\begin{equation}
(f_1 \sqcup f_2)_*(g_1 \sqcup g_2)^* = (f_1)_*(g_1)^* + (f_2)_*(g_2)^*,
\end{equation}
\begin{equation}\label{equ3}
(f_1 \sqcup f_2)_*(td(T_{g_1 \sqcup g_2}) \cap (g_1 \sqcup g_2)^*) = (f_1)_*(td(T_{g_1}) \cap (g_1)^*) + (f_2)_*(td(T_{g_2}) \cap (g_2)^*).
\end{equation}
The proof of these equalities is the same as that of (\ref{equ2}). Here we show only (\ref{equ3}).
\begin{align*}
& (f_1)_*(td(T_{g_1}) \cap (g_1)^*) + (f_2)_*(td(T_{g_2}) \\
& = \Bigl ((f_1)_* pr_1 + (f_2)_*pr_2 \Bigr) \circ \Bigl ( \bigl (td(T_{g_1}) \cap (g_1)^* \bigr ) \oplus \bigl (td(T_{g_2}) \cap (g_2)^* \bigr )  \Bigr )\\
& = \Bigl ((f_1 \sqcup f_2)_* \circ \bigl ((i_1)_* pr_1 + (i_2)_*pr_2 \bigr ) \Bigr) \circ \\
& \hspace{5.5cm} \Bigl (\bigl ((i_1)^* \oplus (i_2)^* \bigr ) \circ \bigl( td(T_{g_1 \sqcup g_2}) \cap (g_1 \sqcup g_2)^* \bigr )\Bigr )\\
& = (f_1 \sqcup f_2)_* \circ \Bigl (\bigl ((i_1)_* pr_1 + (i_2)_*pr_2 \bigr ) \circ ((i_1)^* \oplus (i_2)^* \bigr ) \Bigr ) \circ \\
& \hspace{9cm} \bigl( td(T_{g_1 \sqcup g_2}) \cap (g_1 \sqcup g_2)^* \bigr ) \\
& = (f_1 \sqcup f_2)_* \bigl( td(T_{g_1 \sqcup g_2}) \cap (g_1 \sqcup g_2)^* \bigr ). 
\end{align*}
\end{proof}

\begin{rem} We define the following subcategories of $\op{Corr}(\mathscr V)_{pro-sm}$ and $\mathscr Corr(\mathscr V)_{pro-sm}^+$ :
\begin{align*}
\op{Corr}(\mathscr V)_{pro-id}(X,Y) & := \{X \xleftarrow f Y \xrightarrow {\op{id}_Y} Y \, | \, f \in hom_{\mathscr V}(Y,X) \, \text{proper} \} ,\\
\op{Corr}(\mathscr V)_{id-sm}(X,Y)& :=  \{X \xleftarrow {\op{id}_X} X \xrightarrow g  Y \, | \, g \in hom_{\mathscr V}(X,Y) \, \text{smooth} \},\\
\mathscr  Corr(\mathscr V)^+_{pro-id}(X,Y) & :=  \{[X \xleftarrow f Y \xrightarrow {\op{id}_Y} Y] \, | \, f \in hom_{\mathscr V}(Y,X) \, \text{proper} \},\\
\mathscr  Corr(\mathscr V)^+_{id-sm}(X,Y) & :=  \{[X \xleftarrow {\op{id}_X} X \xrightarrow g  Y] \, | \, g \in hom_{\mathscr V}(X,Y) \, \text{smooth} \}.
\end{align*}
Here we note that 
\begin{align*}
\mathscr  Corr(\mathscr V)^+_{pro-id}(X,Y) & = \op{Corr}(\mathscr V)_{pro-id}^+(X,Y) \\
& = \{[X \xleftarrow f Y ] \, | \, f \in hom_{\mathscr V}(Y,X) \, \text{proper} \},
\end{align*}
\begin{align*}
\mathscr  Corr(\mathscr V)^+_{id-sm}(X,Y) & = \op{Corr}(\mathscr V)_{id-sm}^+(X,Y) \\
& =\{[X \xrightarrow g  Y] \, | \, g \in hom_{\mathscr V}(X,Y) \, \text{smooth} \}.\end{align*}
Then $\mathscr G_0: \mathscr  Corr(\mathscr V)^+_{pro-id} \to \mathscr  Ab$ and $\mathscr  H^{Todd}_*: \mathscr  Corr(\mathscr V)^+_{pro-id} \to \mathscr  Ab$ are covariant functors (in the usual sense) for proper morphisms  and $td_*^{\op{BFM}}: \mathscr  G_0(-) \to \mathscr H^{Todd}_*(-)$ is Baum--Fulton--MacPherson's Todd class transformation. \, $\mathscr G_0:   \mathscr Corr(\mathscr V)^+_{id-sm} \to \mathscr  Ab$  and $\mathscr  H^{Todd}_*: \mathscr  Corr(\mathscr V)^+_{id-sm} \to \mathscr  Ab$ are contravariant functors for smooth morphisms and $td_*^{\op{BFM}}: \mathscr  G_0(-) \to \mathscr H^{Todd}_*(-)$ is the Verdier--Riemann--Roch formula for Baum--Fulton--MacPherson's Todd class transformation.
\end{rem}
\begin{rem}\label{rem-*} The above Verdier-Riemann-Roch formula (\ref{diagram*3}) 
holds for any $\ell.c.i.$ morphism $g:M \to Y$ instead of a smooth morphism $g:M \to Y$. The main reason why we restrict ourselves to smooth morphisms $g:M \to Y$, i.e., considering proper-smooth correspondences instead of proper-$\ell.c.i.$ correspondences $X \xleftarrow f M \xrightarrow g Y$ with proper morphism $f$ and $\ell.c.i.$ morphism $g$ is that a local complete intersection morphism is not necessarily stable under a base change, i.e., in taking the product (see the diagram (\ref{product})
the pull-backed one $\widetilde g:M \times_Y N \to N$ is not necessarily an $\ell.c.i.$ morphism even if $g:M \to Y$ is so. Thus the composite $k \circ \widetilde g: M \times_Y N \to Z$ is not necessarily an $\ell.c.i.$ morphism even if $g:M \to Y$ and $k:N \to Z$ are $\ell.c.i.$ morphisms. We can remedy this drawback, also by considering zigzags instead of correspondences, which we will also discuss later.
\end{rem}
\begin{rem} Given any two kinds of classes of morphisms, we can consider correspondences similar to proper-smooth correspondences.  Let $\mathscr M_1, \mathscr  M_2$ be two classes of morphisms such that they contain all identity morphisms and they are also stable by base change and closed under composition; i.e., 
\begin{enumerate}
\item if $f:X \to Y$ is in the class $\mathscr  M_i$, then for any fiber square
$$
\xymatrix{
X' \ar[r]^{g'} \ar[d]_{f'} & X \ar[d]^f\\
Y' \ar[r]_g & Y
}
$$
$f'$ is in the class $\mathscr  M_i$.
\item if $f:X \to Y$ and $g:Y \to Z$ are in the class $\mathscr M_i$, then the composite $g \circ f$ is also in the class $\mathscr M_i$.
\end{enumerate}
Then a correspondence $X \xleftarrow f V \xrightarrow g Y$ with $f \in \mathscr M_1$ and $g \in \mathscr M_2$ shall be called a \emph{$(\mathscr M_1,\mathscr M_2)$-correspondence} from $X$ to $Y$. Then in a similar manner we can get the category $\op{Corr}(\mathscr V)_{(\mathscr M_1,\mathscr M_2)}$ and the $\mathscr Ab$-enriched category $\mathscr Corr(\mathscr V)^+_{(\mathscr M_1,\mathscr M_2)}$ of $(\mathscr M_1,\mathscr M_2)$-correspondences.
If a functor $E: \mathscr V \to \mathscr Ab$ satisfies that 
\begin{enumerate}
\item the functor $E: \mathscr V \to \mathscr Ab$ is covariant for morphisms in the class $\mathscr M_1$,
\item the functor $E: \mathscr V \to \mathscr Ab$ is contravariant for morphisms in the class $\mathscr M_2$,
\end{enumerate}
then the functor $E: \mathscr V \to \mathscr Ab$ shall be called \emph{a partially bifunctor with respect to $(\mathscr M_1, \mathscr M_2)$}.
Let us suppose that a partially bifunctor $E: \mathscr V \to \mathscr Ab$ with respect to $(\mathscr M_1, \mathscr M_2)$ satisfies 
(the base change formula): 
$$\xymatrix{
A' \ar[d]_{\widetilde h} \ar[r]^{\widetilde g} & A \ar[d]^{h \quad \qquad \Longrightarrow \quad \qquad} \\
B' \ar[r]_g & B,
}
\xymatrix{
E(A') \ar[d]_{(\widetilde h)_*} && E(A) \ar[ll]_{(\widetilde g)^*} \ar[d]^{h_*}\\
E(B') && E(B) \ar[ll]^{g^*}.
}
$$
Here $h \in \mathscr M_1$ and $g \in \mathscr M_2$.
Such a partially bifunctor shall be called \emph{a nice partially bifunctor with respect to $(\mathscr M_1, \mathscr M_2)$}. Let $E_1, E_2:\mathscr V \to \mathscr Ab$ be two nice partially bifunctors with respect to $(\mathscr M_1, \mathscr M_2)$ and let $\tau: E_1 \to E_2$ be a natural transformation. Then, in the same way as above we get the following:
 
Define the functors $\mathcal E_1:\op{Corr}(\mathscr V)_{(\mathscr M_1, \mathscr M_2)} \to \mathscr Ab$ and $\mathcal E_2:\op{Corr}(\mathscr V)_{(\mathscr M_1, \mathscr M_2)} \to \mathscr Ab$ as follows:
\begin{enumerate}
\item For an object $X \in Obj(\op{Corr}(\mathscr V)_{(\mathscr M_1, \mathscr M_2)}) = Obj(\mathscr V)$, 
$$\mathcal E_1(X): =E_1(X), \quad \mathcal E_2(X) :=E_2(X).$$
\item For $X \xleftarrow f M \xrightarrow g Y \in hom_{\op{Corr}(\mathscr V)_{(\mathscr M_1, \mathscr M_2)}}(X,Y) = \op{Corr}(X,Y)_{(\mathscr M_1, \mathscr M_2)}$, 
\begin{align*}
\mathcal E_1(X \xleftarrow f M \xrightarrow g Y) &:=f_*g^*:\mathcal E_1(Y) \to \mathcal E_1(X), \\
\mathcal E_2(X \xleftarrow f M \xrightarrow g Y) &:= f_*g^*:\mathcal E_2(Y) \to \mathcal E_2(X).
\end{align*}
\end{enumerate}
Then $\mathcal E_1:\op{Corr}(\mathscr V)_{(\mathscr M_1, \mathscr M_2)} \to \mathscr Ab$ and $\mathcal E_2:\op{Corr}(\mathscr V)_{(\mathscr M_1, \mathscr M_2)} \to \mathscr Ab$  are 
functors and the natural transformation $\tau: E_1 \to E_2$ is extended to the natural transformation
$\tau: \mathcal  E_1(-) \to \mathcal E_2(-).$

If $E_i$ is additive, then similarly we define $\mathscr E_1: \mathscr  Corr(\mathscr V)^+_{(\mathscr M_1, \mathscr M_2)} \to \mathscr  Ab$ and $\mathscr  E_2: \mathscr  Corr(\mathscr V) ^+_{(\mathscr M_1, \mathscr M_2)} \to \mathscr  Ab$ as follows: 
\begin{enumerate}
\item For an object $X \in Obj(\mathscr  Corr(\mathscr V) ^+_{(\mathscr M_1, \mathscr M_2)})$, 
$$\mathscr  E_1(X):= E_1(X), \quad \mathscr  E_2(X):=E_2(X).$$
\item For a morphism $\sum_in_i[X \xleftarrow {f_i} M_i \xrightarrow {g_i} Y]  \in hom_{\mathscr  Corr(\mathscr V) _{(\mathscr M_1, \mathscr M_2)}^+}(X,Y) := \op{Corr}(X,Y)_{(\mathscr M_1, \mathscr M_2)}^+$ 
\begin{align*}
\mathscr  E_1 \Bigl (\sum_in_i[X \xleftarrow {f_i} M_i \xrightarrow {g_i} Y] \Bigr) &:= \sum_in_i(f_i)_*(g_i)^*:\mathscr  E_1(Y) \to \mathscr  E_1(X),\\
\mathscr  E_2\Bigl (\sum_in_i[X \xleftarrow {f_i} M_i \xrightarrow {g_i} Y] \Bigr) &:= \sum_in_i(f_i)_*(g_i)^* :\mathscr  E_2(Y)\to \mathscr  E_2(X).
\end{align*}
\end{enumerate}
Then $\mathscr E_1: \mathscr  Corr(\mathscr V) ^+_{(\mathscr M_1, \mathscr M_2)} \to \mathscr  Ab$ and $\mathscr  E_2: \mathscr  Corr(\mathscr V) ^+_{(\mathscr M_1, \mathscr M_2)} \to \mathscr  Ab$ are 
functors and the natural transformation $\tau: E_1 \to E_2$ is extended to the natural transformation 
$\tau : \mathscr  E_1(-) \to \mathscr  E_2(-).$
\end{rem}

Baum--Fulton--MacPherson's Riemann--Roch transformation was motivated by MacPherson's Chern class transformation \cite{Mac}, which is the unique natural transformation $c^{\op{Mac}}_*:F(-) \to H_*(-)$ from the covariant functor $F(-)$ of constructible functions to the covariant Borel--Moore homology theory $H_*(-)$, satisfying the ``smooth condition'' that for a smooth variety $X$ the value $c_*^{\op{Mac}}(\jeden_X) = c(TX) \cap [X]$ the Poincar\'e dual of the total Chern class $c(TX)$ of the tangent bundle $TX$
 (see \cite{Fulton-book}). Here $\jeden_X$ is the characteristic function on $X$. 

 For a proper morphism $f:M \to X$ we have the commutative diagram:
\begin{equation}\label{diagram*5}
\xymatrix{
F(X)\ar[d]_{c^{\op{Mac}}_*}  && F(M) \ar[ll]_{f_*} \ar[d]^{c^{\op{Mac}}_*} \\
H_*(X) && H_*(M) \ar[ll]^{f_*}.
}
\end{equation}
For a smooth morphism\footnote{If $g:M \to Y$ is a local complete intersection morphism, this Verdier--Riemann--Roch formula (\ref{diagram*6}) does not hold in general, but there is some defect as proved by J. Sch\"urmann \cite{Sch}.} $g:M \to Y$ we have the following Verdier--Riemann--Roch formula for MacPherson's Chern class transformation \cite{Yokura-VRR}:
\begin{equation}\label{diagram*6}
\xymatrix{
F(M)\ar[d]_{c^{\op{Mac}}_*}  && F(Y) \ar[ll]_{g^*} \ar[d]^{c^{\op{Mac}}_*} \\
H_*(M) && H_*(Y) \ar[ll]^{c(T_g) \cap g^*}.
}
\end{equation}
Here $c(T_g)$ is the total Chern class of the relative tangent bundle $T_g$ of the smooth  morphism $g$.  The pullback $g^*:F(Y) \to F(M)$ is simply the pullback of functions, i.e., for a constructible function $\ga:Y \to \mathbb Z$, $g^*(\ga)$ is defined by $(g^*(\ga))(m):= \ga(g(m))$ for $m \in M$.
Hence combining the above two commutative diagrams (\ref{diagram*5}) and (\ref{diagram*6}), for a correspondence $(X \xleftarrow f M \xrightarrow g Y)$ with proper $f$ and smooth $g$, we have the commutative diagrams:
\begin{equation}\label{diagram*4}
\xymatrix{
F(X)\ar[d]_{c^{\op{Mac}}_*}  && F(M) \ar[ll]_{f_*} \ar[d]_{c^{\op{Mac}}_*} && F(Y) \ar[ll]_{g^*} \ar[d]^{c^{\op{Mac}}_*} \\
H_*(X) && H_*(M) \ar[ll]^{f_*} && H_*(Y) \ar[ll]^{c(T_g) \cap g^*}.
}
\end{equation}
As in the case of Baum--Fulton--MacPherson's Riemann--Roch transformation $td_*^{\op{BFM}}:G_0(X) \to H_*(X)\otimes \mathbb Q$ we obtain the following:
\begin{pro}\label{pro-Mac} Define the functors $\mathcal F:\op{Corr}(\mathscr V)_{pro-sm} \to \mathscr Ab$ and $\mathcal H^{Chern}_*:\op{Corr}(\mathscr V)_{pro-sm} \to \mathscr Ab$ as follows:
\begin{enumerate}
\item For an object $X \in Obj(\op{Corr}(\mathscr V)_{pro-sm}) = Obj(\mathscr V)$, 
$$\mathcal F(X) =F(X), \quad \mathcal H^{Chern}_*(X) :=H_*(X).$$
\item For $X \xleftarrow f M \xrightarrow g Y \in hom_{\op{Corr}(\mathscr V)_{pro-sm}}(X,Y) = \op{Corr}(X,Y)_{pro-sm}$, 
\begin{align*}
\mathcal F(X \xleftarrow f M \xrightarrow g Y) &:=f_*g^*:\mathcal F(Y) \to \mathcal F(X), \\
\mathcal H^{Chern}_*(X \xleftarrow f M \xrightarrow g Y) &:= f_* \bigl(c(T_g)\cap g^* \bigr):\mathcal H^{Chern}_*(Y) \to \mathcal H^{Chern}_*(X).
\end{align*}
\end{enumerate}
Then $\mathcal F:\op{Corr}(\mathscr V)_{pro-sm} \to \mathscr Ab$ and $\mathcal H^{Chern}_*:\op{Corr}(\mathscr V)_{pro-sm} \to \mathscr Ab$  are 
functors in the sense of
$\mathcal F (\alp \circ \be) = \mathcal F(\alp) \circ \mathcal F(\be)$ and $\mathcal H^{Chern}_* (\alp \circ \be) = \mathcal H^{Chern}_*(\alp) \circ \mathcal H^{Chern}_*(\be)$, 
and MacPherson's Chern class transformation $c_*^{\op{Mac}}: F(-) \to H_*(-)$ is extended to the natural transformation
$$c_*^{\op{Mac}}: \mathcal  F(-) \to \mathcal H^{Chern}_*(-).$$
\end{pro}
\begin{thm}\label{thm-Mac}
We define
$$\text{$\mathscr F: \mathscr  Corr(\mathscr V)^+_{pro-sm} \to \mathscr  Ab$ and $\mathscr  H^{Chern}_*: \mathscr  Corr(\mathscr V)^+_{pro-sm} \to \mathscr  Ab$}$$
 as follows: 
\begin{enumerate}
\item For an object $X \in Obj(\mathscr  Corr(\mathscr V)^+_{pro-sm})$, 
$$\text{$\mathscr  F(X):= F(X)$ and $\mathscr  H^{Chern}_*(X):=H_*(X)$,}$$
\item For a morphism $\sum_in_i[X \xleftarrow {f_i} M_i \xrightarrow {g_i} Y]  \in hom_{\mathscr  Corr(\mathscr V)_{pro-sm}^+}(X,Y) := \op{Corr}(X,Y)_{pro-sm}^+$ 
$$\mathscr  F \Bigl (\sum_in_i[X \xleftarrow {f_i} M_i \xrightarrow {g_i} Y] \Bigr) := \sum_in_i(f_i)_*(g_i)^*:\mathscr F(Y) \to \mathscr F(X),$$
\begin{align*}
\mathscr  H^{Chern}_*\Bigl (& \sum_in_i[X \xleftarrow {f_i}  M_i \xrightarrow {g_i} Y] \Bigr) \\
& := \sum_in_i(f_i)_* \Bigl(c(T_{g_i}) \cap (g_i)^* \Bigr):\mathscr  H^{Chern}_*(Y) \to \mathscr  H^{Chern}_*(X).
\end{align*}
\end{enumerate}
Then $\mathscr F: \mathscr  Corr(\mathscr V)^+_{pro-sm} \to \mathscr  Ab$ and $\mathscr  H^{Chern}_*: \mathscr  Corr(\mathscr V)^+_{pro-sm} \to \mathscr  Ab$ are 
functors in the sense of
$\mathscr F (\alp \circ \be) = \mathscr F(\alp) \circ \mathscr F(\be)$ and $\mathscr H^{Chern}_* (\alp \circ \be) = \mathscr H^{Chern}_*(\alp) \circ \mathscr H^{Chern}_*(\be)$,
and MacPherson's Chern class transformation $c_*^{\op{Mac}}: F(-) \to H_*(-)$ is extended to the natural transformation
$$c_*^{\op{Mac}}: \mathscr  F(-) \to \mathscr H^{Chern}_*(-).$$
\end{thm}
\begin{rem} 
$\mathscr F: \mathscr  Corr(\mathscr V)^+_{pro-id} \to \mathscr  Ab$ and $\mathscr  H^{Chern}_*: \mathscr  Corr(\mathscr V)^+_{pro-id} \to \mathscr  Ab$ are covariant functors (in the usual sense) for proper morphisms  and $c_*^{\op{Mac}}: \mathscr  F(-) \to \mathscr H^{Chern}_*(-)$ is MacPherson's Chern class transformation. \, \, $\mathscr F: \mathscr  Corr(\mathscr V)^+_{id-sm} \to \mathscr  Ab$ and $\mathscr  H^{Chern}_*: \mathscr  Corr(\mathscr V)^+_{id-sm} \to \mathscr  Ab$ are contravariant functors for smooth morphisms and $c_*^{\op{Mac}}: \mathscr  F(-) \to \mathscr H^{Chern}_*(-)$ is the Verdier--Riemann--Roch formula for MacPherson's Chern class transformation.
\end{rem}
\begin{rem} $\mathcal F(X \xleftarrow f M \xrightarrow g Y) :=f_*g^*:F(Y) \to F(X)$ is called a topological Radon transformation of constructible functions \cite{Ern, EOY, Scha}. The base change formula for constructible functions is proved in \cite[Proposition 3.5]{Ern} (cf. \cite[Lemma 2.4]{EOY}). $\mathcal H^{Chern}_*(X \xleftarrow f M \xrightarrow g Y) := f_* \bigl(c(T_g)\cap g^* \bigr):H_*(Y)  \to H_*(X)$ for compact smooth manifolds $X,Y,M$ is called the Verdier--Radon transformation in \cite{EOY} (cf. \cite{Yokura-Kagoshima}). Here $f$ and $g$ can be any morphism and $T_g:=TM - g^*TY$ is the virtual relative tangent bundle.
\end{rem}
In \cite{Looi} E. Looijenga defines the relative Grothendieck group $K_0(\mathscr V/X)$ as the free abelian group generated by the isomorphism classes $[V \xrightarrow h X]$ of a morphism $h:V \to X$ modulo the relation 
 $$[V \xrightarrow h X] = [W \xrightarrow {h|_W} X] + [V \setminus W \xrightarrow {h|_{V \setminus W}} X]$$
  for a closed subvariety $W \subset V$.
For a morphism $f:X \to Y$ the pushforward 
$$f_*:K_0(\mathscr V/X) \to K_0(\mathscr V/Y)$$
 is defined by $f_*([V \xrightarrow h X]):=[V \xrightarrow {f \circ h} Y]$ and clearly $(g \circ f)_* =g_* \circ f_*$, namely $K_0(\mathscr V/-)$ is a covariant functor. For a morphism $g:X' \to X$, the pullback 
 $$g^*:K_0(\mathscr V/X) \to K_0(\mathscr V/X')$$
  is defined by $g^*([V \xrightarrow h X]):= [V' \xrightarrow {h'} X']$ where we use the following fiber square:
$$
\xymatrix{
V' \ar[r]^{g'} \ar[d]_{h'}& V \ar[d]^h\\
X' \ar[r]_g & X.
}
$$
Then it is clear that for morphisms $g:X' \to X$ and $f:X'' \to X$ we have $(g \circ f)^* = f^* \circ g^*$, thus it is a contravariant functor. We observe that the functor $K_0(\mathscr V/-)$ satisfies the base change formula:
$$\xymatrix{
A' \ar[d]_{\widetilde h} \ar[r]^{\widetilde g} & A \ar[d]^{h \quad \qquad \Longrightarrow \quad \qquad} \\
B' \ar[r]_g & B,
}
\xymatrix{
K_0(\mathscr V/A') \ar[d]_{(\widetilde h)_*} && K_0(\mathscr V/A) \ar[ll]_{(\widetilde g)^*} \ar[d]^{h_*}\\
K_0(\mathscr V/B') && K_0(\mathscr V/B) \ar[ll]^{g^*}. 
}
$$ 
This follows from considering the following fiber squares: for $[V \xrightarrow f A] \in K_0(\mathscr V/A)$
$$
\xymatrix{
V' \ar[r]^{\widetilde g'} \ar[d]_{f'}& V \ar[d]^f\\
A' \ar[r]^{\widetilde g} \ar[d]_{\widetilde h}& A \ar[d]^h\\
B' \ar[r]_g & B.
}
$$
In \cite{BSY} we showed that there exists a unique natural transformation
$${T_y}_*: K_0(\mathscr V/X) \to H_*(X)\otimes \mathbb Q[y]$$
such that for a nonsingular variety $X$, ${T_y}_*([X \xrightarrow {\op{id}_X} X]) = T_y(TX) \cap [X]$.
Here $T_y(TX)$ is the Hirzebruch class of the tangent bundle $TX$. The Hirzebruch class $T_y(E)$ of a complex vector bundle is 
$$T_{y}(E):= \prod_{i=1}^{\op{dim} X} \Biggl (\frac{\alpha_i(1+y)}{1-e^{-\alpha_i(1+y)}} -\alpha_i y \Biggr)$$
where $\alpha_{i}$ are the Chern roots of the tangent bundle $E$.
\begin{rem} The Hirzebruch class $T_{y}(E)$ unifies the
following three distinguished and important characteristic cohomology classes of $TX$:
\begin{enumerate}
\item $(y=-1)$: $c(TX) = \prod_{i=1}^{\op{dim} X}  (1+\alpha)$ the total Chern class, 
\item $(y=\, \, \, \, 0)$: $td(TX) = \prod_{i=1}^{\op{dim} X} \frac{\alpha}{1-e^{-\alpha}}$ the total Todd class, 
\item $(y= \, \, \, \, 1)$: $ L(TX) = \prod_{i=1}^{\op{dim} X}  \frac{\alpha}{\tanh \alpha}$ the total Thom--Hirzebruch $L$-class. 
\end{enumerate}
\end{rem}
The natural transformation ${T_y}_*: K_0(\mathscr V/X) \to H_*(X)\otimes \bQ[y]$ is called \emph{the motivic Hirzebruch class}. We also have the Verdier--Riemann--Roch formula for the motivic Hirzebruch class for a smooth morphism (\cite{BSY}). Thus for a correspondence $(X \xleftarrow f M \xrightarrow g Y)$ with proper morphism $f$ and smooth morphism $g$, we have the following commutative diagrams:
\begin{equation}\label{diagram4}
\xymatrix{
K_0(\mathscr V/X)\ar[d]_{T_{y*}}  && K_0(\mathscr V/M) \ar[ll]_{f_*} \ar[d]_{T_{y*}} && K_0(\mathscr V/Y) \ar[ll]_{g^*} \ar[d]^{T_{y*}} \\
H_*(X)\otimes \mathbb Q[y] && H_*(M)\otimes \mathbb Q[y] \ar[ll]^{f_*} && H_*(Y)\otimes \mathbb Q[y] \ar[ll]^{T_y(T_g) \cap g^*}.
}
\end{equation}
Here $T_y(T_g)$ is the Hirzebruch class of the relative tangent bundle of the smooth morphism $g:M \to Y$.

Thus, as in the above discussion, we obtain the following:
\begin{pro} Define the functors $\mathcal K_0(\mathscr V/-):\op{Corr}(\mathscr V)_{pro-sm} \to \mathscr Ab$ and $\mathcal H^{Hirz}_*:\op{Corr}(\mathscr V)_{pro-sm} \to \mathscr Ab$ as follows:
\begin{enumerate}
\item For an object $X \in Obj(\op{Corr}(\mathscr V)_{pro-sm}) = Obj(\mathscr V)$, 
$$\mathcal K_0(\mathscr V/ -)(X) =K_0(\mathscr V/X), \quad \mathcal H^{Hirz}_*(X) :=H_*(X)\otimes \mathbb Q[y].$$
\item For $X \xleftarrow f M \xrightarrow g Y \in hom_{\op{Corr}(\mathscr V)_{pro-sm}}(X,Y) = \op{Corr}(X,Y)_{pro-sm}$, 
\begin{align*}
\mathcal K_0(\mathscr V-)(X \xleftarrow f M \xrightarrow g Y) &:=f_*g^*:\mathcal K_0(\mathscr V/ -)(Y) \to \mathcal K_0(\mathscr V/ -)(X),\\
\mathcal H^{Hirz}_*(X \xleftarrow f M \xrightarrow g Y) &:= f_* \Bigl(T_y(T_g)\cap g^* \Bigr):\mathcal H^{Hirz}_*(Y) \to \mathcal H^{Hirz}_*(X).
\end{align*}
\end{enumerate}
Then $\mathcal K_0(\mathscr V-):\op{Corr}(\mathscr V)_{pro-sm} \to \mathscr Ab$ and $\mathcal H^{Hirz}_*:\op{Corr}(\mathscr V)_{pro-sm} \to \mathscr Ab$  are 
functors and the motivic Hirzebruch class transformation ${T_y}_*: K_0(\mathscr V/X) \to H_*(X)\otimes \bQ[y]$ is extended to the natural transformation
$$T_{y*}: \mathcal K_0(\mathscr V-) \to \mathcal H^{Hirz}_*(-).$$
\end{pro}
\begin{thm}
We define $\mathscr K_0(\mathscr V/-): \mathscr  Corr(\mathscr V)^+_{pro-sm} \to \mathscr  Ab$ and $\mathscr  H^{Hirz}_*: \mathscr  Corr(\mathscr V) ^+_{pro-sm} \to \mathscr  Ab$ as follows: 
\begin{enumerate}
\item For an object $X \in Obj(\mathscr  Corr(\mathscr V) ^+_{pro-sm})$, 
$$\text{$\mathscr K_0(\mathscr V/-)(X):= F(X)$ and $\mathscr  H^{Hirz}_*(X):=H_*(X)\otimes \mathbb Q[y]$,}$$
\item For a morphism $\sum_in_i[X \xleftarrow {f_i} M_i \xrightarrow {g_i} Y]  \in hom_{\mathscr  Corr(\mathscr V) _{pro-sm}^+}(X,Y) := \op{Corr}(X,Y)_{pro-sm}^+$ 
\begin{align*}
\mathscr K_0(\mathscr V/-) \Bigl (& \sum_in_i[X \xleftarrow {f_i} M_i \xrightarrow {g_i} Y] \Bigr) \\
&:= \sum_in_i(f_i)_*(g_i)^*:\mathscr K_0(\mathscr V/-)(Y) \to \mathscr K_0(\mathscr V/-)(X), 
\end{align*}
\begin{align*}
\mathscr  H^{Hirz}_*\Bigl (& \sum_in_i[X \xleftarrow {f_i} M_i \xrightarrow {g_i} Y] \Bigr) \\
&:= \sum_in_i(f_i)_* \Bigl(T_y(T_{g_i})\cap (g_i)^* \Bigr):\mathscr  H^{Hirz}_*(Y)  \to \mathscr  H^{Hirz}_*(X).
\end{align*}
\end{enumerate}
Then $\mathscr K_0(\mathscr V/-): \mathscr  Corr(\mathscr V) ^+_{pro-sm} \to \mathscr  Ab$ and $\mathscr  H^{Hirz}_*: \mathscr  Corr(\mathscr V) ^+_{pro-sm} \to \mathscr  Ab$ are 
functors and the motivic Hirzebruch class transformation ${T_y}_*: K_0(\mathscr V/X) \to H_*(X)\otimes \bQ[y]$ is  extended to the natural transformation
$$T_{y*}: \mathscr K_0(\mathscr V/-) \to \mathscr H^{Hirz}_*(-).$$
\end{thm}
\begin{rem} These two functors $\mathscr K_0(\mathscr V/-): \mathscr  Corr(\mathscr V)^+_{pro-id} \to \mathscr  Ab$ and $\mathscr  H^{Hirz}_*: \mathscr  Corr(\mathscr V)^+_{pro-id} \to \mathscr  Ab$ are covariant functors (in the usual sense) for proper morphisms  and $T_{y*}: \mathscr  G_0(-) \to \mathscr H^{Todd}_*(-)$ is the motivic Hirzebruch class transformation. $\mathscr K_0(\mathscr V/-): \mathscr  Corr(\mathscr V)^+_{id-sm} \to \mathscr  Ab$ and $\mathscr  H^{Hirz}_*: \mathscr  Corr(\mathscr V)^+_{id-sm} \to \mathscr  Ab$ are contravariant functors for smooth morphisms and $T_{y*}: \mathscr K_0(\mathscr V/-) \to \mathscr H^{Todd}_*(-)$ becomes the Verdier--Riemann--Roch formula for the motivic Hirzebruch class transformation.
\end{rem}
\begin{rem}
The motivic Hirzebruch class ${T_y}_*: K_0(\mathscr V/X) \to H_*(X)\otimes \bQ[y]$ ``unifies" the above MacPherson's Chern class transformation $c_*^{\op{Mac}}$, Baum--Fulton--MacPherson's Todd class transformation $td_*^{\op{BFM}}$ and Cappell--Shaneson's $L$-class transformation $L_*^{\op{CS}}:\Omega(-) \to H_*(-)\otimes \mathbb Q$ (see below) in the sense that we have the following commutative diagrams:
$$\xymatrix{
& K_0(\Cal V/X)  \ar [dl]_{\epsilon} \ar [dr]^{{T_{-1}}_*} \\
{F(X) } \ar [rr] _{c_*^{\op{Mac}}\otimes \mathbb Q}& &  H_*(X)\otimes \bQ,}
$$
$$\xymatrix{
&  K_0(\Cal V/X)  \ar [dl]_{\Gamma} \ar [dr]^{{T_{0}}_*} \\
{G_0(X) } \ar [rr] _{td_*^{\op{BFM}}}& &  H_*(X)\otimes \bQ,}
$$
$$\xymatrix{
& K_0(\Cal V/X)  \ar [dl]_{\omega} \ar [dr]^{{T_{1}}_*} \\
{\Omega(X) } \ar [rr] _{L_*^{\op{CS}}}& &  H_*(X)\otimes \bQ.}
$$
Goresky--MacPherson's homology $L$-class \cite{GM}, which is extended as a natural transformation by S. Cappell and J. Shaneson \cite{CS} (also see \cite{Yokura-TAMS}):\emph{There exists a unique natural transformation 
$$L_*^{\op{CS}}: \Omega(X) \to  H_*(X)\otimes \bQ$$
such that for a nonsingular compact variety $X$, $L_*(\mathbb Q_X[2 \op{dim}X]) = L(TX) \cap [X]$.}
Here $\Omega$ is the covariant functor assigning to $X$ the cobordism group $\Omega(X)$ of self-dual constructible sheaf complexes on compact $X$.
As for the case of Cappell--Shaneson's $L$-class $L_*^{\op{CS}}:\Omega(-) \to H_*(-)\otimes \mathbb Q$ we do not know whether we can have the Verdier--Riemann--Roch formula for a smooth morphism. Thus, unlike the cases of MacPherson's Chern class, Baum--Fulton--MacPherson's Todd class and the motivic Hirzebruch class, at the moment we cannot define $L_*^{\op{CS}}: \Omega(-)  \to \mathscr  H^{L-class}_*(-)$ in a similar manner to $c^{\op{Mac}}_*: \mathscr  F(-)\to \mathscr  H^{Chern}_*(-)$, $td_*^{\op{BFM}}: \mathscr  G_0(-) \to \mathscr  H^{Todd}_*(-)$ and $T_{y*}: \mathscr K_0(\mathscr V/-) \to \mathscr H^{Hirz}_*(-).$
\end{rem}
\begin{rem} The motivic Chern class transformation $mC_y: K_0(\mathcal V/-) \to G_0(-)\otimes \mathbb Z[y]$ and the motivic Hodge Chern class transformation $MHC_y:K_0(MHM(-)) \to G_0(-)\otimes \mathbb Z[y, y^{-1}]$ also satisfy the Verdier--Riemann--Roch formula for smooth morphisms (see \cite{BSY}). Thus we can get the same formulations for these transformations. Here $K_0(MHM(X))$ is the Grothendieck group of the derived category of mixed Hodge modules on $X$.
\end{rem}
\begin{rem}
In \cite{LP} Levine and Pandharipande show that Levine--Morel's algebraic cobordism $\Omega_*(X)$ can be obtained as a quotient group of the Grothendieck group $\mathcal M^+(X) =\op{Corr}(X,pt)^+_{pro-sm}$ via the double-point degeneration relation. If we can get some quotient group $B\Omega_*(X,Y) :=\widetilde {\op{Corr}(X,Y)^+_{pro-sm}}$ of $\op{Corr}(X,Y)^+_{pro-sm}$ via some analogous manner as in \cite{LP}, in such a way that 
\begin{enumerate}
\item when $Y$ is a point, $B\Omega_*(X,pt) \cong \Omega_*(X)$, 
\item $\circ: \op{Corr}(X,Y)^+_{pro-sm} \times \op{Corr}(Y,Z)^+_{pro-sm} \to \op{Corr}(X,Z)^+_{pro-sm}$ is extended to
$$\circ : B\Omega_*(X,Y) \times B\Omega_*(Y,Z) \to B\Omega_*(X,Z),$$
\end{enumerate}
then we would call $B\Omega_*(X,Y)$ \emph{a bi-variant algebraic cobordism of bicycles}, which is treated in \cite{Yokura-bicycle}. Then we would get an $\mathscr Ab$-enriched category $\mathscr B\Omega_*(\mathscr V)$ of algebraic cobordism of bicycles such that
\begin{enumerate}
\item $Obj(\mathscr B\Omega_*(\mathscr V)) = Obj(\mathscr V)$,
\item For two objects $X$ and $Y$, $hom_{\mathscr B\Omega_*(\mathscr V)}(X,Y) = B\Omega_*(X,Y)$.
\end{enumerate}
and we also could consider whether one can extend characteristic classes of singular varieties to the enriched categories of algebraic cobordism of bicycles. 
\end{rem}

\section{Enriched categories of zigzags and \\ characteristic classes of singular varieties}

A $\ell.c.i$ morphism $f:X \to Y$ is a regular embedding $r:X \to P$ followed by a smooth morphism $p:P \to Y$; $f= p \circ r$ (see \cite{Fulton-book}). For the context of $\ell.c.i$ morphisms we assume as in \cite{Fulton-book} that all varieties have a closed embedding into a smooth variety (e.g. quasi-projective varieties), so that the composition of $\ell.c.i$ morphisms is again a $\ell.c.i$ morphism. The virtual tangent bundle $T_f$ of a $\ell.c.i$ morphism $f:X \to Y$ is defined to be $r^*T_p -N_XP \in K^0(X)$, which has similar properties as the fiber tangent bundle $T_f$ of a smooth morphism $f$ (for more details see \cite{Fulton-book}).

The Verdier--Riemann--Roch formula for Baum--Fulton--MacPherson's Todd class transformation holds for $\ell.c.i$ morphisms (see \cite{BFM}). A smooth morphism is also a $\ell.c.i$ morphism. Hence, similarly we can consider \emph{a proper-$\ell.c.i.$ correspondence} $X \xleftarrow f M  \xrightarrow g Y$ with proper morphism $f$ and $\ell.c.i.$-morphism $g$. Unfortunately, the pullback of an $\ell.c.i.$-morphism in a fiber square is not necessarily an $\ell.c.i.$-morphism. Hence, we cannot do the same argument for proper-$\ell.c.i.$ correspondences. To remedy this drawback, we use \emph{zigzags} instead of correspondences.
\begin{defn}
The following finite sequence of correspondences is called a \emph{$k$-zigzag} or a \emph{$k$-correspondence} of complex algebraic varieties:
{\tiny
$$\xymatrix{
& M_1 \ar[dl]_{f_1} \ar[dr]^{g_1} & & M_2\ar[dl]_{f_2} \ar[dr]^{g_2} & & & M_k\ar[dl]_{f_k} \ar[dr]^{g_k} \\
X=X_0 & &   X_1 & & X_2 \cdots  & \cdots X_{k-1} & &  X_k=Y.
}
$$
}

The set of all zigzags of finite length from $X$ to $Y$ is denoted by $Zigzag(X,Y)$.
\end{defn}
\begin{lem} For two zigzags 
$$\alp =(X \xleftarrow {f_1} M_1 \xrightarrow {g_1} X_1 \cdots \cdots X_{i-1}\xleftarrow {f_i} M_i \xrightarrow {g_i} Y) \in Zigzag(X,Y),$$ 
$$\be =(Y \xleftarrow {h_1} N_1 \xrightarrow {k_1} Y_1 \cdots \cdots Y_{j-1} \xleftarrow {h_j} N_j \xrightarrow {k_j} Z) \in Zigzag(Y,Z),$$ 
we define the product $\alp \wedge \be$ by juxtaposition:
\begin{align*}
& \alp \wedge \be := \\
&(X \xleftarrow {f_1} M_1 \xrightarrow {g_1} X_1 \cdots  X_{i-1}\xleftarrow {f_i} M_i \xrightarrow {g_i} Y \xleftarrow {h_1} N_1 \xrightarrow {k_1} Y_1 \cdots  Y_{j-1} \xleftarrow {h_j} N_j \xrightarrow {k_j} Z).
\end{align*}
Then the juxtaposition $\wedge$ is well-defined:
$$\wedge: Zigzag(X,Y) \times Zigzag(Y,Z) \to Zigzag(X, Z).$$
\end{lem}
If a zigzag consists of proper-$\ell.c.i.$ correspondence, i.e., if each $X_{i-1} \xleftarrow {f_i} M_i \xrightarrow {g_i} X_i$ is a proper-$\ell.c.i.$ correspondence, such a zigzag is called \emph{a proper-$\ell.c.i.$ zigzag}, abusing words, and the set of all proper-$\ell.c.i.$  zigzags from $X$ to $Y$ is denoted by $Zigzag_{pro-\ell.c.i.}(X, Z)$. 

Then we define the category $\op{Zigzag}_{pro-\ell.c.i.}(\mathscr V)$ of proper-$\ell.c.i.$ zigzags:
\begin{itemize}
\item $Obj(\op{Zigzag}_{pro-\ell.c.i.}(\mathscr V)) =Obj(\mathscr V).$
\item For $X$ and $Y$, $hom_{\op{Zigzag}_{pro-\ell.c.i.}(\mathscr V)}(X,Y) = Zigzag_{pro-\ell.c.i.}(X,Y).$
\end{itemize}

Two proper-$\ell.c.i.$ zigzags (of the same length) $(X \xleftarrow {f_1} M_1 \xrightarrow {g_1} X_1 \xleftarrow {f_2} M_2 \xrightarrow {g_2} X_2 \cdots \cdots \xleftarrow {f_k} M_k \xrightarrow {g_k} Y)$ and  $(X=X_0 \xleftarrow {f'_1} M'_1 \xrightarrow {g'_1} X'_1 \xleftarrow {f'_2} M'_2 \xrightarrow {g'_2} X'_2 \cdots \cdots X'_{k-1}\xleftarrow {f'_k} M'_k \xrightarrow {g'_k} X_k=Y)$ are called isomorphic if there exist isomorphisms $h_i:M_i \to M'_i$ for $1 \leqq i \leqq k$ and $\phi_j:X_j \to X'_j$ for $1 \leqq j \leqq k-1$ such that the following diagrams commute
{\small$$\xymatrix{
& M_1 \ar[ddd]_{h_1}^{\cong} \ar[dl]_{f_1} \ar[dr]^{g_1} & & M_2 \ar[ddd]_{h_2}^{\cong} \ar[dl]_{f_2} \ar[dr]^{g_2} & & & M_k \ar[ddd]_{h_k}^{\cong} \ar[dl]_{f_k} \ar[dr]^{g_k} \\
X\ar[d]_{=} & &   X_1 \ar[d]_{\phi_1}^{\cong} & & X_2 \ar[d]_{\phi_2}^{\cong} \cdots  & \cdots X_{k-1}\ar[d]_{\phi_{k-1}}^{\cong}  & &  Y \ar[d]_{=}\\
X & &   X'_1 & & X'_2  \cdots  & \cdots X'_{k-1} & &  Y \\
& M'_1 \ar[ul]_{f_1} \ar[ur]^{g_1} & & M'_2\ar[ul]_{f_2} \ar[ur]^{g_2} & & & M'_k\ar[ul]_{f_k} \ar[ur]^{g_k} \\
}
$$
}
The set of isomorphism classes of proper-$\ell.c.i.$ zigzag (of length $k$) becomes an Abelian monoid by taking the disjoint sum
\begin{align*}
& [X\xleftarrow {f_1} M_1 \xrightarrow {g_1} X_1 \cdots \cdots X_{k-1}\xleftarrow {f_k} M_k \xrightarrow {g_k} Y] \\
& \hspace{4cm} + [X \xleftarrow {f'_1} M'_1 \xrightarrow {g'_1} X'_1  \cdots \cdots X'_{k-1}\xleftarrow {f'_k} M'_k \xrightarrow {g'_k} Y] := \\
& {\tiny [X\xleftarrow {f_1 + f'_1} M_1 \sqcup M'_1 \xrightarrow {g_1 + g'_1} X_1 \sqcup X'_1\cdots X_{k-1} \sqcup X'_{k-1}\xleftarrow {f_k + f'_k} M_k \sqcup M'_k \xrightarrow {g_k +g'_k} Y]}.
\end{align*}
Then the group completion of it is denoted by  $Zigzag_k(X,Y)^+_{pro-\ell.c.i.}$ and 
$$Zigzag(X,Y)^+_{pro-\ell.c.i.}:= \bigoplus_k Zigzag_k(X,Y)^+_{pro-\ell.c.i.}.$$
The product by juxtaposition 
$$\wedge: Zigzag_{pro-\ell.c.i.}(X,Y) \times Zigzag_{pro-\ell.c.i.}(Y,Z) \to Zigzag_{pro-\ell.c.i.}(X, Z)$$
 is extended to $Zigzag_{pro-\ell.c.i.}(-,-)^+$:
$$\wedge : Zigzag(X,Y)^+_{pro-sm}\times Zigzag(Y,Z)^+_{pro-\ell.c.i.} \to Zigzag(X,Z)^+_{pro-\ell.c.i.}.$$
Then the $\mathscr Ab$-enriched category $\mathscr Zigzag(\mathscr V)^+_{pro-\ell.c.i.}$ of zigzags is defined as
\begin{itemize}
\item $Obj(\mathscr Zigzag(\mathscr V)^+_{pro-\ell.c.i.}) = Obj(\mathscr V)$,
\item For $X$ and $Y$, $hom_{\mathscr Zigzag(\mathscr V)^+_{pro-\ell.c.i.}}(X,Y) = Zigzag(X,Y)^+_{pro-\ell.c.i.}.$
\end{itemize}

Then
Proposition \ref{pro-BFM} and Theorem \ref{thm-BFM} become as follows:
\begin{pro}
 Define the functors 
$$\mathcal G_0:\op{Zigzag}_{pro-\ell.c.i.}(\mathscr V) \to \mathscr Ab, \quad \mathcal H^{Todd}_*:\op{Zigzag}_{pro-\ell.c.i.}(\mathscr V) \to \mathscr Ab$$ as follows:
\begin{enumerate}
\item For an object $X \in Obj(\op{Zigzag}_{pro-\ell.c.i.}(\mathscr V)) = Obj(\mathscr V)$, 
$$\mathcal G_0(X) =G_0(X), \, \mathcal H^{Todd}_*(X) :=H_*(X)\otimes \mathbb Q \, \, \text{the Borel--Moore homology theory} $$
\item For a morphism  
\begin{align*}
\alp = (X \xleftarrow {f_1} M_1  \xrightarrow {g_1} & X_1 \cdots  \cdots X_{i-1} \xleftarrow {f_i} M_i \xrightarrow {g_i} Y) \\
& \in hom_{\op{Zigzag}_{pro-\ell.c.i.}(\mathscr V)}(X,Y) = Zigzag_{pro-\ell.c.i}(X,Y)
\end{align*}
$$\mathcal G_0(\alp) := (f_1)_*(g_1)^* \circ \cdots \circ (f_i)_*(g_i)^*:  \mathcal G_0(Y) \to \mathcal G_0(X),$$
\begin{align*}
& \mathcal  H^{Todd} _*(\alp)\\
& :=(f_1)_*(td(T_{g_1})(g_1)^*) \circ \cdots \circ (f_i)_*(td(T_{g_i})(g_i)^*)):  \mathcal H^{Todd}_*(Y) \to \mathcal H^{Todd}_*(X).
\end{align*}
\end{enumerate}
Then $\mathcal G_0:\op{Zigzag}_{pro-\ell.c.i.}(\mathscr V) \to \mathscr Ab$ and $\mathcal H^{Todd}_*:\op{Zigzag}_{pro-\ell.c.i.}(\mathscr V) \to \mathscr Ab$ are 
functors in the sense of
$\mathcal G_0(\alp \wedge \be) = \mathcal G_0(\alp) \circ \mathcal G_0(\be)$ and $ \mathcal H^{Todd}_*(\alp \wedge  \be) = \mathcal  H^{Todd}_*(\alp)\circ \mathcal H^{Todd}_*(\be)$, and
Baum--Fulton--MacPherson's Todd class transformation $td_*^{\op{BFM}}: G_0(-) \to H_*(-)\otimes \mathbb Q$ is extended to the natural transformation
$$td_*^{\op{BFM}}: \mathcal  G_0(-) \to \mathcal H^{Todd}_*(-).$$
\end{pro}
\begin{thm} 
We define 
$$\text{$\mathscr G_0: \mathscr Zigzag(\mathscr V)^+_{pro-\ell.c.i.} \to \mathscr  Ab$ and $\mathscr  H^{Todd}_*: \mathscr Zigzag(\mathscr V)^+_{pro-\ell.c.i.} \to \mathscr  Ab$}$$
 as follows: 
\begin{enumerate}
\item For an object $X \in Obj(\mathscr Zigzag(\mathscr V)^+_{pro-\ell.c.i.})$, $\mathscr  G_0(X):= F(X)$ and $\mathscr  H^{Todd}_*(X):=H_*(X)\otimes \mathbb Q$,
\item For $\sum_in_i[\alp_i]  \in hom_{\mathscr Zigzag(\mathscr V)^+_{pro-\ell.c.i.}}(X,Y) := Zigzag_(X,Y)^+_{pro-\ell.c.i.}$ 
\begin{align*}
\mathscr  G_0 \Bigl (\sum_in_i[\alp_i] \Bigr) &:= \sum_in_i\mathcal  G_0 (\alp_i): \mathscr G_0(Y) \to \mathscr G_0(X),\\
\mathscr  H^{Todd}_*\Bigl ([\alp_i] \Bigr) &:= \sum_in_i\mathcal  H^{Todd}_*(\alp_i): \mathscr  H^{Todd}_*(Y) \to \mathscr  H^{Todd}_*(X).
\end{align*}
\end{enumerate}
Then $\mathscr G_0:\mathscr Zigzag(\mathscr V)^+_{pro-\ell.c.i.} \to \mathscr  Ab$ and $\mathscr  H^{Todd}_*: \mathscr Zigzag(\mathscr V)^+_{pro-\ell.c.i.} \to \mathscr  Ab$ are 
functors in the sense of $\mathscr G_0(\alp \wedge \be) = \mathscr G_0(\alp) \circ \mathscr G_0(\be)$ and $\mathscr H^{Todd}_*(\alp \wedge  \be) = \mathscr  H^{Todd}_*(\alp)\circ \mathscr H^{Todd}_*(\be)$, and MacPherson's Chern class transformation $td_*^{\op{BFM}}: G_0(-) \to H_*(-)\otimes \mathbb Q$ is extended to the natural transformation
$$td_*^{\op{BFM}}: \mathscr  G_0(-) \to \mathscr H^{Todd}_*(-).$$
\end{thm}
\begin{rem} We note the following:
\begin{enumerate}
\item A zigzag of proper-identity correspondences is the same as a proper-identity correspondence $X \xleftarrow f Y \xrightarrow {\op{id}_Y} Y$ with proper $f$,
\item A zigzag of identity-smooth correspondences is the same as an identity-smooth correspondence $X \xleftarrow {\op{id}_X} X \xrightarrow g Y$ with smooth  $g$,
\item A zigzag of identity-$\ell.c.i.$ correspondences is the same as an identity-$\ell.c.i.$ correspondence $X \xleftarrow {\op{id}_X} X \xrightarrow g Y$ with $\ell.c.i.$ morphism $g$.
\end{enumerate}
\end{rem}
\section{Cobordism bicycles of vector bundles and \\ Baum--Fulton--MacPherson's Todd classes}
In this section we consider extending the notion of algebraic cobordism of vector bundles due to Y.-P. Lee and R. Pandharipande \cite{LeeP} to correspondences.

\begin{defn} Let $X \xleftarrow p V \xrightarrow s Y$ be a proper-smooth correspondence and let $E$ be a complex algebraic vector bundle on $V$. Then the pair of them
$$(X \xleftarrow p V \xrightarrow s Y; E)$$
is called a \emph{cobordism bicycle of a vector bundle}.
\end{defn}
\begin{rem} The above cobordism bicycle of a vector bundle $(X \xleftarrow p V \xrightarrow s Y; E)$ is just a proper-smooth correspondence equipped with a  complex vector bundle, but mimicking terminologies used in \cite{BB}, \cite{LM} and \cite{LeeP}, we name it so. A similar object is used in the so-called KK-theory (e.g., see \cite{EM3}). Such KK-theoretic things, i.e., bi-variant theoretic aspects are treated in \cite{Yokura-bicycle}.
\end{rem}
\begin{defn} Let $(X \xleftarrow p V \xrightarrow s Y; E)$ and $(X \xleftarrow {p'} V' \xrightarrow {s'} Y; E')$ be two cobordism bicycles of vector bundles of the same rank. If there exists an isomorphism $h: V \cong V'$ such that $(X \xleftarrow p V \xrightarrow s Y) \cong (X \xleftarrow {p'} V' \xrightarrow {s'} Y)$ as correspondences and $E \cong h^*E'$ as well,
they are called isomorphic and denoted by
$$(X \xleftarrow p V \xrightarrow s Y; E) \cong (X \xleftarrow {p'} V' \xrightarrow {s'} Y; E').$$
\end{defn}
The isomorphism class of a cobordism bicycle of a vector bundle $(X \xleftarrow p V \xrightarrow s Y; E)$ is denoted by $[X \xleftarrow p V \xrightarrow s Y; E]$, which is still called a cobordism bicycle of a vector bundle.  For a fixed rank $r$ for vector bundles, the set of isomorphism classes of cobordism bicycles of vector bundles for a pair $(X,Y)$  becomes a commutative monoid by the disjoint sum:
\begin{align*}
[X \xleftarrow {p_1} V_1 \xrightarrow {s_1} Y; E_1] + [X \xleftarrow {p_2} V_2 & \xrightarrow {s_2} Y; E_2]\\
& := [X \xleftarrow {p_1+p_2} V_1 \sqcup V_2 \xrightarrow {s_1+s_2} Y; E_1 + E_2],
\end{align*}
where $E_1 + E_2$ is a vector bundle such that $(E_1 + E_2)|_{V_1} =E_1$ and $(E_1 + E_2)|_{V_2} =E_2.$
This monoid is denoted by $\mathcal  M_r(X,Y)$ and another grading of $[X \xleftarrow p V \xrightarrow s Y; E]$ is defined by the relative dimension $\op{dim} s$ of the smooth map $s$, thus by double grading, $[X \xleftarrow p V \xrightarrow s Y; E] \in \mathcal  M_{n,r}(X,Y)$ means that $n = \op{dim} s$ and $r = \op{rank} E$.
The group completion of this monoid, i.e., the Grothendieck group, is denoted by $\mathcal  M_{n,r}(X,Y)^+$. We use this notation, mimicking \cite{LM, LeeP} (cf. Remark \ref{M+}).
\begin{rem} For a fixed rank $r$, $\mathcal  M_{*, r}(X,Y)^+ = \bigoplus \mathcal  M_{n,r}(X,Y)^+$ is  a graded abelian group.
\end{rem}
\begin{rem} When $Y=pt$ a point, $\mathcal  M_{n,r}(X,pt)^+$ is nothing but $\mathcal  M_{n,r}(X)^+$ which is considered in Lee--Pandharipande \cite{LeeP}. In this sense, when $X = pt$ a point, $\mathcal  M_{n,r}(pt,Y)^+$ is a new object to be investigated.
\end{rem}

\begin{defn}[product of cobordism bicycles]\label{prod} For three varieties $X,Y,Z$, we define the following two kinds of product
\begin{enumerate}
\item (by the Whitney sum $\oplus$)
$$\circ_{\oplus} : \mathcal  M_{m,r}(X,Y)^+ \times \mathcal  M_{n,k}(Y,Z)^+ \to \mathcal  M_{m+n,r+k}(X,Z)^+, $$
\begin{align*}
[X \xleftarrow p  V \xrightarrow s Y; E] \circ_{\oplus}   [& Y \xleftarrow q  W \xrightarrow t Z; F]\\
&: = [(X \xleftarrow p V \xrightarrow s Y) \circ (Y \xleftarrow q W \xrightarrow s Z);\widetilde{q}^*E \oplus \widetilde{s}^*F ],
\end{align*}
\item (by the tensor product $\otimes$)
$$\circ_{\otimes}  : \mathcal  M_{m,r}(X,Y)^+ \times \mathcal  M_{n,k}(Y,Z)^+ \to \mathcal  M_{m+n,rk}(X,Z)^+,$$
\begin{align*}
[X \xleftarrow p  V \xrightarrow s Y; E] \circ_{\otimes} & [Y \xleftarrow q W \xrightarrow t Z; F]\\
&: = [(X \xleftarrow p V \xrightarrow s Y) \circ (Y \xleftarrow q W \xrightarrow s Z);\widetilde{q}^*E \otimes \widetilde{s}^*F ],
\end{align*}
\end{enumerate}
where we consider the following commutative diagram
$$\xymatrix{
& & \widetilde{q}^*E \oplus \widetilde{s}^*F \quad \text{or} \quad \widetilde{q}^*E \otimes  \widetilde{s}^*F \ar[d] \\
& E \ar[d]& V\times_Y W\ar [dl]_{\widetilde{q}} \ar[dr]^{\widetilde{s}} & F \ar[d]&\\
& V \ar [dl]_{p} \ar [dr]^{s} && W \ar [dl]_{q} \ar[dr]^{t}\\
X & &  Y && Z. }
$$
\end{defn}

\begin{lem}
The products $\circ_{\oplus}$ and $\circ_{\otimes}$ are both bilinear.
\end{lem}
\begin{rem} $\mathcal  M_{*,*}(X,X)^+$ is a double graded commutative ring with respect to both products $\circ_{\oplus}$ and $\circ_{\otimes}$
\end{rem}
\begin{rem} We consider the above product $\circ_{\oplus}$ for $Y = Z=pt$ a point, since $\mathcal  M_{n,r}(X,pt)^+ = \mathcal  M_{n,r}(X)^+$ and $\mathcal  M_{n,r}(pt,pt)^+ = \mathcal  M_{n,r}(pt)^+$, we have
$$\circ_{\oplus}: \mathcal  M_{m,r}(X)^+ \times \mathcal  M_{n,k}(pt)^+ \to \mathcal  M_{m+n,r+k}(X)^+.$$
\begin{align*}
[X \xleftarrow p  V \xrightarrow s pt; E] \circ_{\oplus} & [pt \xleftarrow q W \xrightarrow t pt; F]
=\\
& [(X \xleftarrow p V \xrightarrow s pt) \circ (pt \xleftarrow q W \xrightarrow s pt);(pr_1)^*E \oplus (pr_2)^*F ],
\end{align*}
which is, using the notations used in \cite{LeeP}, rewritten as follows
$$[V \xrightarrow p X, E] \circ_{\oplus} [W; F]= [V \times W \xrightarrow {p \circ pr_1} X; (pr_1)^*E \oplus (pr_2)^*F].$$
$$\xymatrix{
& & (pr_1)^*E \circ_{\oplus} (pr_2)^*F \ar[d] \\
& E \ar[d]& V\times W\ar [dl]_{pr_1} \ar[dr]^{pr_2} & F \ar[d]&\\
& V \ar [dl]_{p} \ar [dr]^{s} && W \ar [dl]_{q} \ar[dr]^{t}\\
X & &  pt && pt. }
$$
\end{rem}

\begin{rem} Let $[X \xleftarrow p  V \xrightarrow s Y; E] \in \mathcal  M_{m,r}(X,Y)^+_{\oplus}$, and let
\begin{enumerate}
\item $[X \xleftarrow p V] := [X \xleftarrow p V \xrightarrow {\op{id}_V} V] \in \mathcal  M_{0,0}(X,V)^+_{\oplus}$,
\item $[V; E] := [V \xleftarrow {\op{id}_V} V \xrightarrow {\op{id}_V} V; E] \in \mathcal  M_{0,r}(V,V)^+_{\oplus}$,
\item $[V \xrightarrow s Y]:= [V \xleftarrow {\op{id}_V} V \xrightarrow s Y] \in \mathcal  M_{m,0}(V,Y)^+_{\oplus}$.
\end{enumerate}
Then we have $[X \xleftarrow p  V \xrightarrow s Y; E] = [X \xleftarrow p V]\circ_{\oplus} [V; E] \circ_{\oplus} [V \xrightarrow s Y].$
\end{rem}
\begin{rem} As to cobordism bicycles, as ``bicycle" suggest, we can discuss bivariant-theoretic aspects, but in this paper we will not discuss them in details anymore. For more detailed properties and bivariant-theoretic aspects, see \cite{Yokura-bicycle}.
\end{rem} 

\begin{defn} Define the enriched category $\mathscr M_{*,*}(\mathscr V)^+_{\oplus}$, $\mathscr M_{*,*}(\mathscr V)^+_{\otimes}$ of cobordism bicycles of vector bundles as follows:
\begin{enumerate}
\item $Obj(\mathscr M_{*,*}(\mathscr V)^+_{\oplus})=Obj(\mathscr M_{*,*}(\mathscr V)^+_{\otimes})=Obj(\mathscr V).$
\item For two objects $X$ and $Y$, 
$$hom_{\mathscr M_{*,*}(\mathscr V)^+_{\oplus}}(X,Y) := \mathcal M_{*,*}(X,Y)^+_{\oplus}, \quad hom_{\mathscr M_{*,*}(\mathscr V)^+_{\otimes}}(X,Y) := \mathcal M_{*,*}(X,Y)^+_{\otimes}.$$
\end{enumerate}
\end{defn}

 \begin{pro}\label{pro-M} Let $c\ell$ be a multiplicative characteristic class of complex vector bundles (hence, in particular, $c\ell(F_1 \oplus F_2) = c\ell(F_1) \cup c\ell(F_2)$) with $c\ell(E) \in H^*(X) \otimes \Lambda$ with some ring $\Lambda$. Define $\underline {\mathscr  H}^{c\ell}_*: \mathscr  M_{*,*}(\mathscr V)^+_{\oplus} \to \mathscr Ab$ by
 \begin{enumerate}
 \item For an object $X$, $\underline {\mathscr  H}^{c\ell}_*(X) = H_*(X)\otimes \Lambda$, the Borel--Moore homology theory with coefficients in $\Lambda$
 \item For a morphism $\sum_i n_i[X \xleftarrow {p_i} V_i \xrightarrow {s_i} Y; E_i] \in hom_{\mathscr M_{*,*}(\mathscr V)^+_{\oplus}}(X,Y)$
 \begin{align*}
 \underline {\mathscr  H}^{c\ell}_*(\sum_i n_i[& X \xleftarrow {p_i} V_i \xrightarrow {s_i} Y; E_i]) \\
 & := \sum_i n_i(p_i)_* \Bigl(c\ell(E_i) \cap (s_i)^*\Bigr): \underline {\mathscr  H}^{c\ell}_*(Y) \to \underline {\mathscr  H}^{c\ell}_*(X).
 \end{align*}
 \end{enumerate}
 Then the functor $\underline {\mathscr  H}^{c\ell}_*: \mathscr  M_{*,*}(\mathscr V)^+_{\oplus}\to \mathscr Ab$ is a 
 functor in the sense of
 $\underline {\mathscr  H}^{c\ell}_*(\alp \circ_{\oplus} \be) = \underline {\mathscr  H}^{c\ell}_*(\alp) \circ \underline {\mathscr  H}^{c\ell}_*(\be).$
\end{pro}
\begin{proof} It suffices to show 
\begin{align*}
\underline {\mathscr  H}^{c\ell}_* ( [X \xleftarrow p V \xrightarrow s Y; E] & \circ_{\oplus}  [Y \xleftarrow q W \xrightarrow t Z; F]) \\
& = \underline {\mathscr  H}^{c\ell}_* ([X \xleftarrow p  V \xrightarrow s Y; E]) \circ \underline {\mathscr  H}^{c\ell}_* ([Y \xleftarrow q W \xrightarrow t Z; F])
\end{align*}
The proof is the same as that of Proposition \ref{pro-BFM}, but for the sake of readers' convenience we write it down.
\begin{align*}
\underline {\mathscr  H}^{c\ell}_* ( [& X \xleftarrow p V \xrightarrow s Y; E]  \circ_{\oplus}  [Y \xleftarrow q W \xrightarrow t Z; F]) \\
& = \underline {\mathscr  H}^{c\ell}_* ([X \xleftarrow {p \circ \widetilde q} V \times _Y W \xrightarrow {t \circ \widetilde s} Z;\widetilde{q}^*E \oplus \widetilde{s}^*F ])\\
& = (p \circ \widetilde q)_* \Bigr (c\ell (\widetilde {q}^*E \oplus \widetilde {s}^*F) \cap (t \circ \widetilde s)^* \Bigr) \\
& = p_* \widetilde q_* \biggl ( \Bigl (\widetilde {q}^*c\ell(E) \cup \widetilde {s}^*c\ell(F) \Bigr) \cap (\widetilde s^* \circ t^*) \biggr) \\
& = p_* \widetilde q_* \biggl ( \widetilde {q}^*c\ell(E) \cap \Bigl ((\widetilde {s}^*c\ell(F) \cap (\widetilde s^* \circ t^*) \Bigr) \biggr) \\
& = p_* \biggl (c\ell(E) \cap \widetilde q_* \Bigl (\widetilde {s}^* (c\ell(F) \cap t^*) \Bigr) \biggr)  \quad \text{(by the projection formula)}\\
& = p_* \biggl (c\ell(E) \cap \widetilde q_* \Bigl ((\widetilde {s}^*c\ell(F) \cap (\widetilde s^* \circ t^*) \Bigr) \biggr) \\
& = p_* \biggl (c\ell(E) \cap \widetilde q_* \widetilde {s}^* (c\ell(F) \cap t^*) \biggr) \\
& = p_* \biggl (c\ell(E) \cap s^*q_* (c\ell(F) \cap t^*) \biggr)  \quad \text{(since $\widetilde q_* \widetilde {s}^* = s^* q_*$)} \\
& = p_* \biggl (c\ell(E) \cap s^* \Bigl (q_* (c\ell(F) \cap t^*) \Bigr) \biggr) \\
& = \underline {\mathscr  H}^{c\ell}_* ([X \xleftarrow p  V \xrightarrow s Y; E]) \circ \underline {\mathscr  H}^{c\ell}_* ([Y \xleftarrow q W \xrightarrow t Z; F]).
\end{align*}
\end{proof}
As a corollary of the proof of the above proposition we get the following for $\mathscr  M_{*,*}(\mathscr V)^+_{\otimes}$, since 
$ch(E \otimes F)= ch(E) \cup ch(F)$ for the Chern character $ch$:
\begin{cor} Let $ch$ be the Chern character. Define $\underline {\mathscr  H}^{ch}_*: \mathscr  M_{*,*}(X,Y)^+_{\otimes} \to \mathscr Ab$ by
 \begin{enumerate}
 \item For an object $X$, $\underline {\mathscr  H}^{ch}_*(X) = H_*(X)\otimes \mathbb Q$ is the Borel--Moore homology theory.
 \item For a morphism $\sum_i n_i[X \xleftarrow {p_i} V_i \xrightarrow {s_i} Y; E_i] \in hom_{\mathscr M_{*,*}(\mathscr V)^+_{\otimes}}(X,Y)$
\begin{align*}
\underline {\mathscr  H}^{ch}_*(& \sum_i n_i[X \xleftarrow {p_i} V_i \xrightarrow {s_i} Y; E_i]) \\
& := \sum_i n_i(p_i)_* \Bigl(ch(E_i) \cap (s_i)^*\Bigr): \underline {\mathscr  H}^{ch}_*(Y)  \to \underline {\mathscr  H}^{ch}_*(X) .
\end{align*}

 \end{enumerate}
 Then the functor $\underline {\mathscr  H}^{ch}_*: \mathscr  M_{*,*}(\mathscr V)^+_{\otimes} \to \mathscr Ab$ is a 
 functor in the sense of
 $\underline {\mathscr  H}^{ch}_*(\alp \circ_{\otimes} \be) = \underline {\mathscr  H}^{ch}_*(\alp) \circ \underline {\mathscr  H}^{ch}_*(\be).$
\end{cor}
Since a smooth map $s:V \to Y$ has the relative tangent bundle $T_s$, we can make another functor as follows.
 \begin{pro}
 For two multiplicative characteristic classes $c\ell_1, c\ell_2$ (with coefficients in a ring $\Lambda$) of complex algebraic vector bundles, we define
 $\mathscr  H^{c\ell_1, c\ell_2}_*: \mathscr  M_{*,*}(\mathscr V)^+_{\oplus} \to \mathscr Ab$ by 
 \begin{enumerate}
 \item For an object $X$, $\mathscr  H^{c\ell_1, c\ell_2}_*(X) = H_*(X)\otimes \Lambda$ is the Borel--Moore homology theory  with coefficients in $\Lambda$
 \item For a morphism $\sum_i n_i[X \xleftarrow {p_i} V_i \xrightarrow {s_i} Y; E_i] \in hom_{\mathscr M_{*,*}(\mathscr V)^+_{\oplus}}(X,Y)$
 \begin{align*}
 \mathscr  H^{c\ell_1, c\ell_2}_*&( \sum_i n_i[X \xleftarrow {p_i}  V_i \xrightarrow {s_i} Y; E_i]) := \\
 & \sum_i n_i(p_i)_* \Bigl(c\ell_1(T_{s_i}) \cap c\ell_2(E_i) \cap (s_i)^*\Bigr):\mathscr  H^{c\ell_1, c\ell_2}_*(Y) \to \mathscr  H^{c\ell_1, c\ell_2}_*(X).
 \end{align*}
 \end{enumerate}
 Then 
 $\mathscr  H^{c\ell_1, c\ell_2}_*: \mathcal  M_{*,*}(\mathscr V)^+_{\oplus} \to \mathscr Ab$ is a 
 functor in the sense of
$\mathscr  H^{c\ell_1, c\ell_2}_*(\alp \circ_{\oplus} \be) = \mathscr  H^{c\ell_1, c\ell_2}_*(\alp) \circ \mathscr  H^{c\ell_1, c\ell_2}_*(\be).$
\end{pro}
\begin{proof}  It suffices to show that
\begin{align*}
\mathscr  H^{c\ell_1, c\ell_2}_*([X \xleftarrow p V & \xrightarrow s Y; E] \circ_{\oplus}  [Y \xleftarrow q W \xrightarrow t Z; F]) \\
& = \mathscr  H^{c\ell_1, c\ell_2}_*([X \xleftarrow p  V \xrightarrow s Y; E]) \circ \mathscr  H^{c\ell_1, c\ell_2}_*([Y \xleftarrow q W \xrightarrow t Z; F]).
\end{align*}
The proof is the same as that of Proposition \ref{pro-M}, but for the sake of readers' convenience we write it down.
\begin{align*}
& \mathscr  H^{c\ell_1, c\ell_2}_*  ([X \xleftarrow p V \xrightarrow s Y; E] \circ_{\oplus}  [Y \xleftarrow q W \xrightarrow t Z; F]) \\
& = \mathcal  H_{c\ell_1, c\ell_2}([X \xleftarrow {p \circ \widetilde q} V \times _Y W \xrightarrow {t \circ \widetilde s} Z;\widetilde{q}^*E \oplus \widetilde{s}^*F ])\\
& = (p \circ \widetilde q)_* \Bigr (c\ell_1(T_{t \circ \widetilde s}) \cap c\ell_2 (\widetilde {q}^*E \oplus \widetilde {s}^*F) \cap (t \circ \widetilde s)^* \Bigr) \\
& = p_* \widetilde q_*\Bigr (\bigl( c\ell_1(T_{\widetilde s}) \cup c\ell_1(\widetilde s^*T_t) \bigr) \cap c\ell_2 (\widetilde {q}^*E \oplus \widetilde {s}^*F) \cap (\widetilde s^* \circ t^*)  \Bigr) \\
& \hspace{2.5cm} \quad \text{(since $c\ell_1$ is multipicative, thus $c\ell_1(T_{t \circ \widetilde s}) = c\ell_1(T_{\widetilde s}) \cup c\ell_1(\widetilde s^*T_t)$)} \\
& = p_* \widetilde q_* \Bigr ((\widetilde q^*c\ell_1(T_s) \cup \widetilde s^*c\ell_1(T_t)) \cap (\widetilde {q}^*c\ell_2 (E) \cup \widetilde {s}^*c\ell_2(F)) \cap (\widetilde s^* \circ t^*)  \Bigr)\\
& \hspace{4.5cm} \quad \text{(since $T_{\widetilde s}= \widetilde q^*T_s$ and $c\ell_2$ is also multiplicative)} \\
& = p_* \widetilde q_* \Bigr (\widetilde q^*c\ell_1(T_s) \cap \widetilde s^*c\ell_1(T_t) \cap \widetilde {q}^*c\ell_2 (E) \cap \widetilde {s}^*c\ell_2(F) \cap (\widetilde s^* \circ t^*)  \Bigr)\\
& = p_* \widetilde q_* \Bigr (\widetilde q^*c\ell_1(T_s) \cap \widetilde {q}^*c\ell_2 (E) \cap \widetilde s^*c\ell_1(T_t) \cap \widetilde {s}^*c\ell_2(F) \cap (\widetilde s^* \circ t^*)  \Bigr)\\
& = p_* \widetilde q_* \Bigr (\widetilde q^* \bigl(c\ell_1(T_s) \cup c\ell_2 (E) \bigr) \cap \widetilde s^* \bigl (c\ell_1(T_t) \cup c\ell_2(F) \bigr)  \cap (\widetilde s^* \circ t^*)  \Bigr)
\end{align*}
\begin{align*}
& = p_* \biggr (\bigl(c\ell_1(T_s) \cup c\ell_2 (E) \bigr) \cap \widetilde q_* \Bigl (\widetilde s^* \bigl (c\ell_1(T_t) \cup c\ell_2(F) \bigr)  \cap (\widetilde s^* \circ t^*) \Bigr)  \biggr) \\
& \hspace{8cm} \quad \text{(by the projection formula)}\\
& = p_* \biggr (\bigl(c\ell_1(T_s) \cup c\ell_2 (E) \bigr) \cap \widetilde q_* \Bigl (\widetilde s^* \Bigl (\bigl (c\ell_1(T_t) \cup c\ell_2(F) \bigr)  \cap t^* \Bigr)  \biggr) \\
& = p_* \biggr (\bigl(c\ell_1(T_s) \cup c\ell_2 (E) \bigr) \cap \widetilde q_* \widetilde s^* \Bigl (\bigl (c\ell_1(T_t) \cup c\ell_2(F) \bigr)  \cap t^* \Bigr) \biggr) \\
& = p_* \biggr (\bigl(c\ell_1(T_s) \cup c\ell_2 (E) \bigr) \cap s^* q_* \Bigl (\bigl (c\ell_1(T_t) \cup c\ell_2(F) \bigr)  \cap t^* \Bigr) \biggr) \\
& \hspace{9.5cm} \quad \text{(since $\widetilde q_* \widetilde {s}^* = s^* q_*$)}\\
& = p_* \Biggr (\bigl(c\ell_1(T_s) \cup c\ell_2 (E) \bigr) \cap s^* \biggl (q_* \Bigl (\bigl (c\ell_1(T_t) \cup c\ell_2(F) \bigr)  \cap t^* \Bigr) \biggr) \Biggr)\\
& = \mathscr  H^{c\ell_1, c\ell_2}_*([X \xleftarrow p  V \xrightarrow s Y; E]) \circ \mathscr  H^{c\ell_1, c\ell_2}_*([Y \xleftarrow q W \xrightarrow t Z; F])
\end{align*}
\end{proof}
As a corollary of the proof of the above proposition we get the following for $\mathcal  M_{m, r}(X, Y)^+_{\otimes}$:
\begin{cor} For a multiplicative characteristic classes $c\ell$ (with rational coefficients) of complex vector bundles and the Chern character $ch$, we define $\mathscr  H^{c\ell, ch}_*: \mathscr  M_{*,*}(\mathscr V)^+_{\otimes} \to \mathscr Ab$ by
 \begin{enumerate}
 \item For an object $X$, $\mathscr  H^{c\ell, ch}_*(X) = H_*(X)\otimes \mathbb Q$ is the Borel--Moore homology theory with rational coefficients.
 \item For a morphism $\sum_i n_i[X \xleftarrow {p_i} V_i \xrightarrow {s_i} Y; E_i] \in hom_{\mathscr M_{*,*}(\mathscr V)^+_{\otimes}}(X,Y)$
 \begin{align*}
 & \mathscr  H^{c\ell, ch}_*( \sum_i n_i[ X \xleftarrow {p_i} V_i \xrightarrow {s_i} Y; E_i]) := \\
 & \hspace{0.5cm} \sum_i n_i(p_i)_* \Bigl(c\ell(T_{s_i}) \cap ch(E_i) \cap (s_i)^*\Bigr): \mathscr  H^{c\ell, ch}_*(Y) \to \mathscr  H^{c\ell, ch}_*(X).
 \end{align*}
 \end{enumerate}
 Then 
 $\mathscr  H^{c\ell, ch}_*: \mathscr  M_{*,*}(\mathscr V)^+_{\otimes} \to \mathscr Ab$ is a 
 functor in the sense of
 $\mathscr  H^{c\ell, ch}_*(\alp \circ_{\otimes} \be) = \mathscr  H^{c\ell, ch}_*(\alp) \circ \mathscr  H^{c\ell, ch}_*(\be).$
\end{cor}
\begin{thm} 
\begin{enumerate}
\item Let us define $\mathscr G_0^{\otimes}: \mathscr M_{*,*}(\mathscr V)^+_{\otimes} \to \mathscr Ab$ by
\begin{enumerate}
\item For an object $X$, $\mathscr G_0^{\otimes}(X) =G_0(X)$,
\item For a morphism $\sum_i n_i[X \xleftarrow {p_i} V_i \xrightarrow {s_i} Y; E_i] \in hom_{\mathscr M_{*,*}(\mathscr V)^+_{\otimes}}(X,Y) :=\mathcal  M_{*,*}(X,Y)^+_{\otimes}$,
\begin{align*}
& \mathscr G_0^{\otimes}(\sum_i n_i[ X \xleftarrow {p_i} V_i \xrightarrow {s_i} Y; E_i]) :=\\
& \hspace{1cm}  \sum_i n_i (p_i)_* \Bigl ([E_i] \otimes (s_i)^* \Bigr):\mathscr G_0^{\otimes}(Y) \to \mathscr G_0^{\otimes}(X).
\end{align*}
\end{enumerate}
Then $\mathscr G_0^{\otimes}: \mathscr M_{*,*}(\mathscr V) ^+_{\otimes} \to \mathscr Ab$ is a 
functor in the sense of 
$\mathscr G_0^{\otimes}(\alp \circ_{\otimes} \be) = \mathscr G_0^{\otimes}(\alp)  \circ \mathscr G_0^{\otimes}(\be).$
\item For the Todd class $td$ and the Chern character $ch$, we define $\mathscr H^{td, ch}_*: \mathscr M_{*,*}(\mathscr V)^+_{\otimes} \to \mathscr Ab$ by
\begin{enumerate}
\item $\mathscr H^{td, ch}_*(X) := H_*(X)\otimes \mathbb Q$,
\item For 
$\sum_i n_i[X \xleftarrow {p_i} V_i \xrightarrow {s_i} Y; E_i] \in hom_{\mathscr M_{*,*}(\mathscr V)^+_{\otimes}}(X,Y)=\mathcal  M_{*,*}(X,Y)^+_{\otimes},$
 \begin{align*}
& \mathscr H^{td, ch}_* \bigl ( \sum_i n_i[X \xleftarrow {p_i}  V_i \xrightarrow {s_i} Y; E_i] \bigr):= \\
& \sum_i n_i {p_i}_*\bigl(td(T_{s_i}) \cap ch(E_i) \cap (s_i)^* \bigr):\mathscr H^{td, ch}_*(Y)  \to \mathscr H^{td, ch}_*(X). 
 \end{align*}
\end{enumerate}
Then $\mathscr H^{td, ch}_*: \mathscr M_{*,*}(\mathscr V)^+_{\otimes} \to \mathscr Ab$ is a 
functor in the sense of 
$\mathscr H^{td, ch}_*(\alp \circ_{\otimes} \be) = \mathscr H_{td, ch}(\alp) \circ \mathscr H_{td, ch}(\be).$
\item  Baum--Fulton--MacPherson's Todd class transformation $td_*^{\op{BFM}}$ gives rise to the natural transformation of these two 
functors $\mathscr G_0^{\otimes}: \mathscr M_{*,*}(\mathscr V)^+_{\otimes} \to \mathscr Ab$ and $\mathscr H^{td, ch}_*: \mathscr M_{*,*}(\mathscr V)^+_{\otimes} \to \mathscr Ab$:
$$td_*^{\op{BFM}}: \mathscr G_0^{\otimes}(-) \to \mathscr H^{td, ch}_*(-).$$
\end{enumerate}
\end{thm}
\begin{proof} It suffices to show that $td_*^{\op{BFM}}: \mathscr G_0^{\otimes} \to \mathscr H^{td, ch}_*$ is a natural transformation, i.e., the following diagram commutes for $(X \xleftarrow p V \xrightarrow s Y; E)$:
\begin{equation*}
\xymatrix{
G_0(X)\ar[d]_{td_*^{\op{BFM}}}  && G_0(V) \ar[ll]_{p_*} \ar[d]_{td_*^{\op{BFM}}} &&& G_0(Y) \ar[lll]_{[E] \otimes s^*} \ar[d]^{td_*^{\op{BFM}}} \\
H_*(X)\otimes \mathbb Q && H_*(V)\otimes \mathbb Q \ar[ll]^{p_*} &&& H_*(Y)\otimes \mathbb Q \ar[lll]^{td(T_s) \cap ch(E) \cap s^*}
}
\end{equation*}
In other words it suffices to show the commutativity of the square on the right hand side, i.e., for an element $\theta \in G_0(Y)$
$$
td_*^{\op{BFM}}([E] \otimes s^*(\theta)) = td(T_s) \cap ch(E) \cap s^*(td_*^{\op{BFM}}(\theta)).
$$
It follows from \cite[Theorem, p.119]{BFM} (see also \cite[Theorem 18.2, (2) Module]{Fulton-book}) that for any class $\be \in K^0(V)$ the $K$-theory of complex algebraic vector bundles, and any element $\alp \in G_0(V)$, we have
\begin{equation}\label{equ1}
td_*^{\op{BFM}}(\be \otimes \alp) = ch(\be) \cap td_*^{\op{BFM}}(\alp).
\end{equation}
Hence 
\begin{align*}
& td_*^{\op{BFM}}([E] \otimes s^*(\theta)) \\
& = ch(E) \cap td_*^{\op{BFM}}(s^*(\theta)) \quad \text{(by (\ref{equ1}) )} \\
& = ch(E) \cap td(T_s) \cap s^*(td_*^{\op{BFM}}(\theta))) \, \, \text {(by the Verdier--Riemann--Roch formula)}\\
& = td(T_s) \cap ch(E) \cap s^*(td_*^{\op{BFM}}(\theta))).
\end{align*}
\end{proof}
\begin{rem}
The above natural transformation $td_*^{\op{BFM}}: \mathscr G_0^{\otimes}(-) \to \mathscr H^{td, ch}_*(-)$ is a $\mathscr M_{*,*}(\mathscr V)^+_{\otimes}$-version of the natural transformation
$td_*^{\op{BFM}}: \mathscr  G_0(-) \to \mathscr  H^{Todd}_*(-)$ of two 
functors $\mathscr G_0: \mathscr  Corr^+_{pro-sm} \to \mathscr  Ab$ and $\mathscr  H^{Todd}_*: \mathscr  Corr^+_{pro-sm} \to \mathscr  Ab$.
\end{rem}
We define the operations of pushforward and pullback of cobordism bicycles of vector bundles for later use. We can of course discuss plausible or natural relations among the operations of product, pushforward and pullback of cobordism bicycles of complex algebraic vector bundles, but they are treated in \cite{Yokura-bicycle}.
\begin{defn}\label{push-pull}
\begin{enumerate} 
\item (Pushforward) 
\begin{enumerate}
\item For a \emph{proper} map $f:X \to X'$, $f_*:\mathcal  M_{m,r}(X,Y)^+ \to \mathcal  M_{m,r}(X',Y)^+$ is defined by
$$f_*([X \xleftarrow{p} V  \xrightarrow{s} Y; E]):= [X' \xleftarrow{f \circ p} V  \xrightarrow{s} Y; E].$$
\item For a \emph{smooth} map $g:Y \to Y'$, 
$$g_*:\mathcal  M_{m,r}(X,Y)^+ \to \mathcal  M_{m+ \op{dim}g,r}(X,Y')^+$$ is defined by
$$g_*([X \xleftarrow{p} V  \xrightarrow{s} Y; E]):= [X \xleftarrow{p} V  \xrightarrow{g \circ s} Y'; E].$$
(Note that $m = \op{dim} s$ and $\op{dim} (g \circ s) = \op{dim}s + \op{dim}g = m + \op{dim}g$.)
\end{enumerate}

\item (Pullback) 
\begin{enumerate}
\item For a \emph{smooth} map $f:X' \to X$, 
$$f^*:\mathcal  M_{m,r}(X,Y)^+ \to \mathcal  M_{m+\op{dim}f,r}(X',Y)^+$$
 is defined by
$$f^*([X \xleftarrow{p} V  \xrightarrow{s} Y; E]) := [X' \xleftarrow{p'} X' \times_X V \xrightarrow{s \circ f'} Y; (f')^*E].$$
Here we consider the following commutative diagram:
$$\xymatrix{
& (f')^*E \ar[d] \ar[r] & E \ar[d] \\
& X' \times_X V \ar[dl]_{p'} \ar[r]^{\qquad f'} & V \ar [dl]^{p} \ar [dr]^{s}\\
X' \ar[r]_f & X &  & Y }
$$
(Note that the left diamond is a fiber square, thus $f':X'\times_X V \to V$ is smooth and $p':X'\times_X V \to X'$ is proper. Note that $\op{dim} f' = \op{dim}$ and $\op{dim}(s \circ f') = \op{dim} s + \op{dim}f' = m + \op{dim} f$.)
\item For a \emph{proper} map $g:Y' \to Y$, $g^*:\mathcal  M_{m,r}(X,Y)^+ \to \mathcal  M_{m,r}(X,Y')^+$ is defined by
$$ g^*([X \xleftarrow{p} V  \xrightarrow{s} Y; E]) := [X \xleftarrow{p \circ g'} V \times _Y Y' \xrightarrow{s'} Y'; (g')^*E].$$
Here we consider the following commutative diagram:
$$\xymatrix{
& E \ar[d] & (g')^*E \ar[l] \ar[d] \\
& V \ar [dl]^{p} \ar [dr]_{s} & V \times _Y Y' \ar[l]_{g' \quad } \ar[dr]^{s'}\\
X &  & Y & Y' \ar[l]^g}
$$
(Note that the right diamond is a fiber square, thus $s':V\times_Y Y'\to Y'$ is smooth and $g':V\times_Y Y'\to V$ is proper, and $\op{dim} s = \op{dim} s'$.)
\end{enumerate}
\end{enumerate}
\end{defn}
\begin{rem} We remark that when we deal with a smooth map $f$ or $g$, both in pushforward and pullback, the first grading is added by the relative dimension $\op{dim} f$ or $\op{dim} g$, but that when we deal with proper maps, the first grading is not changed. In both pushforward and pullback, the second grading (referring to the dimension of vector bundle) is not changed.
\end{rem}

Let $f:X \to Y$ be a proper and smooth map. Then it follows that we have the pushforward $f_*:\mathcal  M_{m,r}(X,X)^+ \to \mathcal  M_{m,r}(Y,X)^+$ for proper $f$ and the pushforward $f_*:\mathcal  M_{m,r}(Y,X)^+ \to \mathcal  M_{m+ \op{dim} f,r}(Y,Y)^+$ for smooth $f$. 
The composition 
$f_* \circ f_*:\mathcal  M_{m,r}(X,X)^+ \to \mathcal  M_{m+\op{dim}f,r}(Y,Y)^+$ is a pushforward, denoted by $f_{**}$. Namely we have
$f_{**}([X \xleftarrow p V \xrightarrow s X; E]) = [Y \xleftarrow {f \circ p} V \xrightarrow {f \circ s} Y; E].$
This is clearly covariantly functorial for proper and smooth maps.

Similarly, it follows that for a proper map $f$ we have the pullback  $f^*:\mathcal  M_{m,r}(Y,Y)^+ \to \mathcal  M_{m,r}(Y,X)^+$ and for a smooth map $f$ the pullback $f^*:\mathcal  M_{m,r}(Y,X)^+ \to \mathcal  M_{m+ \op{dim} f,r}(Y,Y)^+$. Then the composition of these two pullbacks $f^* \circ f^*:\mathcal  M_{m,r}(X,X)^+ \to \mathcal  M_{m,r}(Y,Y)^+$ is the pullback, denoted by $f^{**}$. Namely we have
$f^{**}([X \xleftarrow p V \xrightarrow s X; E]) = [X \xleftarrow {p' \circ f''} V' \times _V V'' \xrightarrow {s' \circ \widetilde f} X; (\widehat f\times_V f')^*E].$
Here we consider the following fiber squares:
{\tiny
$$
\xymatrix{
& &  & (\widehat f\times_V f')^*E \ar[dl]\\
&  & \qquad (X \times_Y V) \times_V (V \times_Y X) \ar[dl]_{f''} \ar[drr]^{\widetilde f}  \ar[dd]^{\widehat f\times_V f'}  & \\
& X \times_Y V \ar[dl]_{p'} \ar[dr]^{\widehat f} & & E \ar[dl]  &   V \times_Y X \ar[dll]^{f'} \ar[dr]^{s'}\\
X \ar[dr]_f & & V \ar[dl]^{p} \ar[dr]_s  && &  X \ar[dll]_f\\
& Y & &  Y.  
}
$$
}
Here we should note that $X \times_Y V$ on the left is $X \times_{f=p} V$ and $V \times_Y X$ on the right is $X \times_{s=f} V$, thus they are different.
Hence $(X \times_Y V) \times_V (V \times_Y X)$ is $(X \times_{f=p} V) \times_V (X \times_{s=f} V)$.
\begin{pro}\label{pro-sm} Let $f:X \to Y$ be a proper and smooth morphism and let $c\ell_1, c\ell_2$ be two multiplicative characteristic classes of complex algebraic vector bundles. Then we have the following commutative diagrams:
\begin{enumerate}
\item
$$\xymatrix{
\mathcal  M_{m,r}(X,X)^+ \ar[r]^{f_{**}} \ar[d]_{\mathcal  H^{c\ell_1, c\ell_2}_*}  &  \mathcal  M_{m+\op{dim}f,r}(Y,Y)^+  \ar[d]^{\mathcal  H^{c\ell_1, c\ell_2}_*} \\
Hom \Bigl(H_*(X), H_*(X) \Bigr)  \ar[r]_{f_{\bigstar \bigstar}} & Hom \Bigl(H_*(Y), H_*(Y) \Bigr)}
$$
where $f_{\bigstar \bigstar}:Hom \Bigl(H_*(X), H_*(X) \Bigr) \to Hom \Bigl(H_*(Y), H_*(Y) \Bigr)$ is 

defined by
$$f_{\bigstar \bigstar}(\mathcal  H):= f_* \circ \mathcal  H \circ (c\ell_1(T_f) \cap f^*),  \quad \mathcal  H \in Hom \Bigl(H_*(X), H_*(X) \Bigr).$$
\item
$$\xymatrix{
\mathcal  M_{m,r}(Y,Y)^+ \ar[r]^{f^{**}} \ar[d]_{\mathcal  H^{c\ell_1, c\ell_2}_*}  &  \mathcal  M_{m+\op{dim}f,r}(X,X)^+  \ar[d]^{\mathcal  H^{c\ell_1, c\ell_2}_*} \\
Hom \Bigl(H_*(Y), H_*(Y) \Bigr)  \ar[r]_{f^{\bigstar \bigstar}} & Hom \Bigl(H_*(X), H_*(X) \Bigr)}
$$
where $f^{\bigstar \bigstar}:Hom \Bigl(H_*(Y), H_*(Y) \Bigr) \to Hom \Bigl(H_*(X), H_*(X) \Bigr)$ is 

defined by
$$f^{\bigstar \bigstar}(\mathcal  H):= (c\ell_1(T_f) \cap f^*)  \circ \mathcal  H \circ f_*,  \quad \mathcal  H \in Hom \Bigl(H_*(Y), H_*(Y) \Bigr)$$
\end{enumerate}
\end{pro} 
\begin{proof} 
We just show the second one. Let $[Y \xleftarrow p V \xrightarrow s Y;E] \in \mathcal  M_{m,r}(Y,Y)^+$.
\begin{align*}
& \mathcal  H_{c\ell_1, c\ell_2} \Bigl (f^{**}([Y \xleftarrow p V \xrightarrow s Y;E]) \Bigr) \\
& =  \mathcal  H_{c\ell_1, c\ell_2} ([X \xleftarrow {p' \circ f''} V' \times _V V'' \xrightarrow {s' \circ \widetilde f} X; (\widehat f\times_V f')^*E])\\
& = (p' \circ f'')_* \Bigl (c\ell_1(T_{s' \circ \widetilde f}) \cap c\ell_2 \bigl ((\widehat f\times_V f')^*E \bigr) \cap (s' \circ \widetilde f)^* \Bigr)\\
& = (p')_* (f'')_* \Bigl (c\ell_1((f'')^*T_{s \circ \widehat f}) \cap c\ell_2 \bigl ((\widehat f\times_V f')^*E \bigr) \cap (s' \circ \widetilde f)^* \Bigr) \\
& \hspace{8cm} \quad \text{(since $T_{s' \circ \widetilde f} = (f'')^*T_{s \circ \widehat f}$)} \\
& = (p')_* (f'')_* \Bigl ((f'')^*c\ell_1(T_{s \circ \widehat f}) \cap c\ell_2 \bigl (f' \circ \widetilde f)^*E \bigr) \cap (s' \circ \widetilde f)^* \Bigr) \\
& \hspace{8cm} \quad \text{(since $\widehat f\times_V f'= f' \circ \widetilde f$)} \\
& = (p')_* (f'')_* \Bigl ((f'')^*c\ell_1(T_{s \circ \widehat f}) \cap (\widetilde f)^* (f')^*c\ell_2(E) \cap \widetilde f^* (s')^* \Bigr) \\
& = (p')_*\biggl (c\ell_1(T_{s \circ \widehat f}) \cap (f'')_* (\widetilde f)^* \Bigl ((f')^*c\ell_2(E) \cap (s')^* \Bigr )  \biggr) \\
& \hspace{7.5cm} \quad \text{(by the projection formula)} \\
& = (p')_*\biggl (c\ell_1(T_{s \circ \widehat f}) \cap (\widehat f)^* f'_* \Bigl ((f')^*c\ell_2(E) \cap (s')^* \Bigr )  \biggr) \\
& \hspace{8cm} 
\quad \text{(since $(f'')_* (\widetilde f)^* = (\widehat f)^* f'_*$ )} \\
& = (p')_*\biggl (c\ell_1(T_{s \circ \widehat f}) \cap (\widehat f)^* \Bigl (c\ell_2(E) \cap f'_*(s')^* \Bigr )  \biggr) 
\quad \text{(by the projection formula)} \\
\end{align*}
\begin{align*}
& = (p')_*\biggl (\bigl( c\ell_1(T_{\widehat f}) \cup c\ell_1((\widehat f)^*T_s) \bigr) \cap (\widehat f)^* \Bigl (c\ell_2(E) \cap f'_*(s')^* \Bigr )  \biggr) \\
& \hspace{1.7cm} \text{(since $c\ell_1$ is multiplicative, thus $c\ell_1(T_{s \circ \widehat f})= c\ell_1(T_{\widehat f}) \cup c\ell_1((\widehat f)^*T_s)$ )}\\
& = (p')_*\biggl ((p')^*c\ell_1(T_f) \cap (\widehat f)^* \Bigl (c\ell_1(T_s) \cap c\ell_2(E) \cap f'_*(s')^* \Bigr )   \biggr) \\
& \hspace {9cm} \text{(since $T_{\widehat f}= (p')^*T_f$)} \\
& = c\ell_1(T_f) \cap (p')_*(\widehat f)^* \Bigl (c\ell_1(T_s) \cap c\ell_2(E) \cap f'_*(s')^* \Bigr ) \, \text{(by the projection formula)}  \\
& = c\ell_1(T_f) \cap f^* p_* \Bigl (c\ell_1(T_s) \cap c\ell_2(E) \cap s^*f_* \Bigr ) \\
& \hspace{5.5cm}  \quad \text{(since $(p')_*(\widehat f)^*= f^*p_*$ and 
$f'_*(s')^*=s^*f_*$)}  \\
& = c\ell_1(T_f) \cap f^* \biggl (p_* \Bigl (c\ell_1(T_s) \cap c\ell_2(E) \cap s^* \Bigr ) \biggr) f_*   \\
& = f^{\bigstar \bigstar} \Bigl (\mathcal  H_{c\ell_1, c\ell_2}([Y \xleftarrow p V \xrightarrow s Y;E]) \Bigr).
\end{align*}
Therefore we get that $\mathcal  H^{c\ell_1, c\ell_2}_* \circ f^{**} = f^{\bigstar \bigstar} \circ \mathcal  H^{c\ell_1, c\ell_2}_*.$
\end{proof}
\begin{rem} Finally we remark that given a cobordism bicycle of vector bundle $[X \xleftarrow p V \xrightarrow s Y; E]$, we can consider a canonical functor of Fourier--Mukai type on derived categories of coherent sheaves. Let $D^b\op{Coh}(X)$ denote the derived category of bounded complexes of coherent sheaves on $X$. Then we have the following functor of Fourier--Mukai
$$\mathcal  H([X \xleftarrow p V \xrightarrow s Y; E]):= p_*([T_s] \otimes [E]\otimes s^*): D^b\op{Coh}(Y) \to D^b\op{Coh}(X).$$
Here a vector bundle is considered as a locally free sheaf, thus a coherent sheaf. 
We will treat this aspect in a different paper. Here we just remark that the $D^b\op{Coh}$-analogue of the above Proposition \ref{pro-sm} is as follows:
{\small 
$$\xymatrix{
\mathcal  M_{m,r}(X,X)^+ \ar[r]^{f_{**}} \ar[d]_{\mathcal  H}  &  \mathcal  M_{m+\op{dim}f,r}(Y,Y)^+  \ar[d]^{\mathcal  H} \\
Functor \Bigl(D^b\op{Coh}(X), D^b\op{Coh}(X) \Bigr)  \ar[r]_{f_{\bigstar \bigstar}} & Functor \Bigl(D^b\op{Coh}(Y), D^b\op{Coh}(Y) \Bigr)}
$$
}
where 
{\small 
$$f_{\bigstar \bigstar}:Functor ( D^b\op{Coh}(X), D^b\op{Coh}(X) \Bigr) \to Functor \Bigl(D^b\op{Coh}(Y), D^b\op{Coh}(Y) \Bigr)$$ 
}
is defined by
$$f_{\bigstar \bigstar}(\mathcal  H):= f_* \circ \mathcal  H \circ ([T_f] \otimes  f^*),  \quad \mathcal  H \in Functor \Bigl(D^b\op{Coh}(X), D^b\op{Coh}(X) \Bigr).$$
\item
{\small 
$$\xymatrix{
\mathcal  M_{m,r}(Y,Y)^+ \ar[r]^{f^{**}} \ar[d]_{\mathcal  H}  &  \mathcal  M_{m+\op{dim}f,r}(X,X)^+  \ar[d]^{\mathcal  H} \\
Functor \Bigl(D^b\op{Coh}(Y), D^b\op{Coh}(Y) \Bigr)  \ar[r]_{f^{\bigstar \bigstar}} & Functor \Bigl(D^b\op{Coh}(X), D^b\op{Coh}(X) \Bigr)}
$$
}
where 
{\small
$$f^{\bigstar \bigstar}:Functor \Bigl(D^b\op{Coh}(Y), D^b\op{Coh}(Y) \Bigr) \to Functor \Bigl(D^b\op{Coh}(X), D^b\op{Coh}(X) \Bigr)$$
}
 is defined by
$$f^{\bigstar \bigstar}(\mathcal  H):= ([T_f] \otimes f^*)  \circ \mathcal  H \circ f_*,  \quad \mathcal  H \in Functor \Bigl(D^b\op{Coh}(Y), D^b\op{Coh}(Y) \Bigr).$$
\end{rem}

\subsection*{Acknowledgements}
The author would like to thank J\"org Sch\"urmann and the anonymous referee for careful reading of the paper and for their valuable comments and constructive suggestions. This work is supported by JSPS KAKENHI Grant Numbers JP16H03936 and JP19K03468.


\end{document}